\documentclass[12pt,leqno]{amsart}
\usepackage{amsthm}
\usepackage{a4}
\usepackage{amsmath}
\usepackage{amssymb}
\usepackage{amscd}
\usepackage{stmaryrd}
\usepackage{tikz}
\usepackage{tikz-cd}
\usepackage{lipsum}
\usepackage{adjustbox}
\usepackage{url}
\usepackage{graphicx}
\usetikzlibrary{matrix}
\usetikzlibrary{cd}
\usepackage{siunitx}
\usepackage{mathrsfs}
\usepackage{pgf,tikz}
\usepackage{mathrsfs}
\usepackage{enumerate}
\usepackage{soul}
\usetikzlibrary{arrows}

\usepackage[hidelinks]{hyperref}
\usepackage{cleveref}
\title[The $u$-invariant of function fields]{The $u$-invariant of function fields in one variable}

\author{Karim Johannes Becher, Nicolas Daans, Vler\"e Mehmeti}

\newtheorem{thm}{Theorem}[section]
\newtheorem{lem}[thm]{Lemma}
\newtheorem{prop}[thm]{Proposition}
\newtheorem{qu}[thm]{Question}
\newtheorem{cor}[thm]{Corollary}

\newtheorem*{cor*}{Corollary}
\newtheorem*{thm*}{Theorem}
\newtheorem{thmI}{Theorem}

\theoremstyle{definition}

\newtheorem*{defn*}{Definition}

\newtheorem{rem}[thm]{Remark}
\newtheorem*{rem*}{Remark}
\newtheorem{ex}[thm]{Example}
\newtheorem*{ex*}{Example}

\newtheorem*{ack*}{Acknowledgements}

\newcommand{\Q}{\mathbb{Q}}

\newcommand{\Z}{\mathbb{Z}}

\newcommand{\zz}{\mathbb{Z}}

\newcommand{\mc}{\mathcal}
\newcommand{\mf}{\mathfrak}
\newcommand{\ovl}{\overline}
\newcommand{\car}{\mathsf{char}}

\renewcommand{\geq}{\geqslant}
\renewcommand{\leq}{\leqslant}
\newcommand{\su}{\widehat{u}}
\newcommand{\half}{\mbox{$\frac{1}2$}}
\newcommand{\mg}[1]{{#1}^\times}
\newcommand{\sq}[1]{{#1}^{\times 2}}
\newcommand{\mfm}{\mf m}
\newcommand{\qq}{\mathbb{Q}}
\newcommand{\rrk}{\mathsf{rrk}}
\renewcommand{\dim}{\mathsf{dim}}
\renewcommand{\det}{\mathsf{det}}

\newcommand{\ff}{\mathbb{F}}
\newcommand{\la}{\langle}
\newcommand{\ra}{\rangle}

\renewcommand{\sup}{\mathsf{sup}}
\renewcommand{\min}{\mathsf{min}}
\newcommand{\nat}{\mathbb{N}}
\newcommand{\rk}{\mathsf{rk}}
\renewcommand{\bmod}{\,\,\mathsf{mod}\,\,}
\newcommand{\rr}{\mathbb{R}}
\newcommand{\cc}{\mathbb{C}}
\newcommand{\hh}{\mathbb{H}}
\renewcommand{\setminus}{\smallsetminus}
\renewcommand{\bmod}{\,\,\mathsf{mod}\,\,}

\newcommand{\scg}[1]{\mg{#1}/\sq{#1}}
\newcommand{\tr}{\mathsf{t}}
\newcommand{\matr}[1]{\mathbb{M}_{#1}}
\newcommand{\sos}[1]{{\mathsf{\Sigma}{#1}^{(2)}}}

\makeatletter
\newcommand{\bigperp}{%
  \mathop{\mathpalette\bigp@rp\relax}%
  \displaylimits
}

\newcommand{\bigp@rp}[2]{%
  \vcenter{    \m@th\hbox{\scalebox{\ifx#1\displaystyle2.1\else1.5\fi}{$#1\perp$}}
  }%
}
\makeatother

\address{(K.J.B.) University of Antwerp, Department of Mathematics, Middelheim\-laan~1, 2020 Antwerpen, Belgium}
\address{(N.D.) Charles University, Faculty of Mathematics and Physics, Department of Algebra, Sokolov\-sk\' a 83, 18600 Praha~8, Czech Republic}
\address{(N.D.) KU Leuven, Department of Mathematics, Celestijnenlaan 200B, 3001 Heverlee, Belgium}
\address{(V.M.) Sorbonne Université and Universit\'{e} Paris Cit\'{e}, CNRS, IMJ-PRG, F-75005 Paris, France}
\address{(V.M.) DMA, École normale supérieure, Université PSL, CNRS, 75005 Paris, France \newline}

\email{karimjohannes.becher@uantwerpen.be}
\email{nicolas.daans@kuleuven.be}
\email{vlere.mehmeti@imj-prg.fr}

\date{14.08.2025}

\begin{document}
\begin{abstract}
The $u$-invariant of a field is the largest
dimension of an anisotropic quadratic torsion form over the field. 
In this article we obtain a bound 
on the $u$-invariant of function fields in one variable over a henselian valued field with arbitrary value group 
and with residue field of characteristic different from~$2$.
This generalises a theorem due to Harbater, Hartmann and Krashen and its extension due to Scheiderer. Their result covers the special case where the valuation is discrete.
We further give a new proof of a theorem due to Parimala and Suresh bounding by $8$ the $u$-invariant of a function field in one variable over any henselian discretely valued field of characteristic $0$ with perfect residue field of characteristic $2$.


\medskip\noindent
{\sc Classification (MSC 2020):} 11E04, 11E81, 12D15, 12J10, 14H05
\end{abstract}

\maketitle

\section{Introduction}

The study of numerical field invariants in quadratic form theory was initiated by Irving Kaplansky in 1953.
In \cite{Kap53},  
he introduces three quantities in $\nat\cup\{\infty\}$ attached to an arbitrary field of characteristic different from $2$ and suggests to study and compare them.
One of them was later called the \emph{$u$-invariant}, a term which goes back to Hanfried Lenz' article \cite{Len63} from 1963.

A field $K$ is called \emph{nonreal} if  $-1$ is a sum of squares in $K$, and otherwise $K$ is called \emph{real}.
For this introduction, let $K$ denote a nonreal field of characteristic different from~$2$.

A quadratic form over~$K$ is called \emph{universal} if it represents all elements of $K$.
Following Kaplansky and Lenz, the \emph{$u$-invariant of $K$} is defined as the smallest natural number $n$ such that any regular quadratic form in $n$ variables over $K$ is universal provided that such a natural number exists, and~$\infty$ otherwise.
This quantity in $\nat\cup\{\infty\}$ is denoted by $u(K)$.

In the early 1970ies, Richard Elman and Tsit-Yuen Lam continued the study of this invariant in a series of articles.
They modified the definition in a way that makes no difference for nonreal fields, but makes it more meaningful for real fields. We will recall their definition in \Cref{S:genu}.

As pointed out in \cite[Theorem~3]{Kap53}, one always has $u(K(\!(t)\!))=2u(K)$, where $K(\!(t)\!)$ denotes the field of Laurent series in $t$ over $K$.
Consequently, and since $u(\cc)=1$, examples of fields of $u$-invariant $2^r$ for any given $r\in\nat$ are easy to obtain.
Incited by this observation, Kaplansky wonders in \cite{Kap53} whether $u(K)$ is always a power of $2$ or $\infty$.
This question was answered in the negative in 1989 by Alexander~Merkurjev. 
After giving in \cite{Mer89} a first construction of a field of $u$-invariant $6$, he showed in \cite{Mer91} that every positive even integer is the $u$-invariant of some field.
Merkurjev's construction of examples involves direct limits and produces fields which are not finitely generated over any proper subfield.
When restricting to fields which are of finite transcendence degree over some familiar base field, no examples are known with a $u$-invariant that is not a power of $2$.

Let $n\in\nat$ and let $F/K$ be a field extension of transcendence degree $n$.
One may try to determine or at least bound the $u$-invariant of $F$ in terms of $n$ and  properties of $K$. 
Well-known statements of this sort are the following:
\begin{itemize}
    \item If $K$ is algebraically closed, then $u(F)\leq 2^n$.
    \item If $K$ is finite, then $u(F)\leq 2^{n+1}$.
    \item If $K$ is a complete discretely valued field with finite residue field, then $u(F)\leq 2^{n+2}$.
\end{itemize}

The first two bounds were established in 1936 by Chiungtze C.~Tsen \cite{Tsen36} and independently in 1952 by Serge Lang \cite[Theorem 6, Corollary]{Lan52}.
The last bound was proven by David Leep  \cite{Lee13} in 2013, based crucially on a result on systems of quadratic forms due to Roger Heath-Brown \cite{HB10}. (When $\car(K)\neq 0$, that bound also follows from  
\cite[Theorem 6 and Theorem 8]{Lan52}.)

In each of the three cases above, the upper bound becomes an equality  if~$F/K$ is finitely generated. This follows by a classical valuation argument; see \Cref{fg-u-lowerbound} below.
Hence, in those situations, we actually have $u(F)=2^d$ where $d$ is the cohomological dimension of $F$.
This might suggest that $u(F)=2^{n+2}$ would hold as well when $K$ is a nonreal number field and $F/K$ a finitely generated extension of transcendence degree $n$.
However, even for $n=1$ this is still an open problem.
A more general open question is whether $u(F)$ can be bounded in terms of $u(K)$ and the transcendence degree of~$F/K$.

We will now focus on function fields in one variable over a given base field.
By a \emph{function field in one variable over $K$}, we mean a finitely generated field extension of transcendence degree $1$.
We can ask for sufficient conditions on $K$ for a certain number $u$ to serve as a common upper bound on $u(F)$ for all function fields in one variable $F/K$.

Our main result provides a transfer principle for this common upper bound between a  henselian valued field $K$ and its residue field, assuming the latter to be of characteristic different from $2$.
In particular, for any odd prime number~$p$, applying \Cref{T:main} below to the field of $p$-adic numbers~$\qq_p$ with its natural valuation, one retrieves the result from \cite{PS10} and \cite{HHK09} that $u(F)=8$ for every function field in one variable $F/\qq_p$.

Our standard reference for the theory of valued fields is \cite{EP05}.
For a valuation~$v$ on $K$, we denote 
by $Kv$ the residue field and by $vK$ the value group of $v$, and we call $v$ \emph{discrete} if $vK\simeq \zz$.
Given a valuation $v$ on $K$, we call the pair $(K,v)$ a \emph{valued field}.
A valuation $v$ on $K$ is called \emph{henselian} if, 
for every finite field extension $L/K$, there is a unique extension of $v$ to a valuation $w$ on $L$;
note that in this case $w$ is henselian, too.
By \cite[Theorem 1.3.1]{EP05}, every complete rank-$1$ valuation is henselian.

The characteristic of a ring $R$ is denoted by $\car(R)$.
Our main result can be stated as follows:

\begin{thmI}[\Cref{T:main-body}]\label{T:main}
Assume that $(K,v)$ is a henselian valued field with $\car(Kv)\neq 2$.
For $u\in\nat$, the following are equivalent:
\begin{enumerate}[$(i)$]
    \item $u(F)\leq [vK:2vK]\cdot u$ for every function field in one variable $F/K$.
    \item $u(E)\leq u$ for every function field in one variable $E/Kv$ or $[vK:2vK]=\infty$.
\end{enumerate}
Furthermore, if $u(E)=u$ for every function field in one variable $E/Kv$, then $u(F)=[vK:2vK]\cdot u$ for every function field in one variable $F/K$. 
\end{thmI}

If $v$ is a discrete valuation, then $[vK:2vK]=2$.
Hence \Cref{T:main} generalises \cite[Theorem 4.10]{HHK09}, where the equivalence of $(i)$ and $(ii)$ was obtained for the case where $(K,v)$ is a complete discretely valued field with $\car(Kv)\neq 2$.
Under the same hypothesis on $(K,v)$,
a valuation theoretic local-global principle for isotropy of quadratic forms over function fields in one variable over $K$
was established in \cite[Cor.~3.5]{CTPS12}, from which \cite[Theorem 4.10]{HHK09} is easily retrieved.
This local-global principle \cite[Theorem 3.1]{CTPS12} was extended in \cite[Cor.~3.19]{Meh19} by the last author of the present article to the situation where the valuation $v$ on $K$ is complete and the value group $vK$ is an arbitrary ordered subgroup of~$\rr$.
It is illustrated in \cite[Section 6]{Meh19} how this local-global principle applies to establish upper bounds on the $u$-invariant of function fields in one variable over complete valued fields with value group contained in~$\rr$.
By \cite[Theorem~4.4]{BDGMZ}, the local-global principle from \cite[Cor.~3.19]{Meh19} extends to the case where $v$ is henselian but not necessarily complete,  keeping the assumption that $vK$ is contained in $\rr$. 
Nevertheless, this version of the local-global principle turns out to be applicable to the study of field invariants in the case of function fields in one variable over henselian valued fields without assumption on the value group, by using reduction arguments concerning the rank of the valuation.
Such reduction steps have been applied in \cite{BDGMZ} to the study of the Pythagoras number of function fields, and in \cite{Daa24} to the study of linkage of Pfister forms over function fields. Here we apply them to the study of the $u$-invariant, to prove \Cref{T:main}.

Unlike any previous results in the same direction,  \Cref{T:main} holds for an arbitrary value group $vK$. 
Independently from our work, a weaker version of the implication from $(ii)$ to $(i)$ was obtained  in \cite[Theorem 4.1]{Man24} for the case where $vK\subseteq \rr$ (see also \Cref{R:compare-Mehmeti19}).

In any case, the applicability of a local-global principle to a problem on field invariants in quadratic form theory depends on the control over the same problem in the local situation, here specifically on the ability to bound the $u$-invariant of a henselian valued field in terms of the $u$-invariant of the residue field.
The study of quadratic forms over a complete valued field with residue characteristic different from $2$ and value group contained in $\rr$ goes back to William Hetherington~Durfee~\cite{Dur48}.
Some of the results were  reproven for complete discrete valuations by Tonny Albert Springer \cite{Spr55}.
The results from \cite{Dur48} extend to henselian valuations without any condition on the value group.
While these statements and techniques are essentially well-known, to the best of our knowledge the literature does not cover them in the desirable generality or under minimal hypotheses.
We therefore include in \Cref{S:DT} a self-contained and elementary presentation of those statements which are relevant to our topic. We indicate in \Cref{rem:Durfee} how this compares to the existing literature.

Next to removing any restricting hypothesis on the value group, another aspect giving \Cref{T:main} a wider scope compared to preceding results on $u$-invariants of function fields is that it applies for the definition of the `general $u$-invariant' introduced by Elman and Lam \cite{EL73}, which is meaningful also for real fields. 
In this setting, Claus Scheiderer had already extended \cite[Cor.~4.12]{HHK09} to \cite[Theorem~3]{Schei09}, thus covering \Cref{T:main} in the case of a discrete valuation.
In \Cref{S:genu} we revisit this definition in a way that applies to all fields, also regardless of their characteristic, and we prove some basic bounds for valued fields.
In dealing with the $u$-invariant of function fields in one variable, the `strong $u$-invariant' of a field was introduced in \cite[Sect.~5]{HHK09} as a short-hand notation.
In \Cref{S:genu-strong}, we briefly revisit this notion, basing it on the general $u$-invariant as to fit the study of real and nonreal function fields of any characteristic.

In \Cref{S:last}, we present the proof of \Cref{T:main} and then provide some examples of fields where it applies.
The obvious challenge left by \Cref{T:main}
is that of removing the restriction on the residue field to be of characteristic different from~$2$.
A valuation whose residue field has characteristic $2$ is also called \emph{dyadic}.
In \Cref{S:dyadic} we do a first step on the problem in the dyadic case.
For the local-global principle from~\cite{Meh19}, this is perfectly possible, as is explained in \cite[Theorem~7.1]{BDD23}.
The study of quadratic forms over the completion 
$F^w$ of the function field~$F$ with respect to a dyadic valuation $w$ on $F$, however, poses additional problems.
The $u$-invariant of dyadic complete discretely valued fields was studied in \cite{MMW91}.
However, the valuations $w$ appearing in the local-global principle are not all discrete.
Using ground work from \cite{ET11} on quadratic forms over dyadic henselian valued fields, we achieve an alternative proof of the following theorem due to Raman~Parimala and Suresh Venapally from \cite{PS14}.

\begin{thmI}[\Cref{PS} or {\cite[Theorem~4.7]{PS14}}]
        Let $(K,v)$ be a henselian discretely valued field such that $\car(K)=0$, $\car(Kv)=2$ and  $Kv$ is perfect.
        Let $F/K$ be a function field in one variable.
        Then $u(F)=8$ if $K$ has a finite field extension of even degree, and otherwise $u(F)=4$.
\end{thmI}

\section{Valuations and torsion-free abelian groups} 

\label{S:TFAG-val}

In the following statements we consider torsion-free abelian groups.
We will denote the group operation additively.

Consider an abelian group $G$. The value $\rrk(G) = \dim_{\Q}(G \otimes_{\Z} \Q)$ in $\nat\cup\{\infty\}$ is called the \emph{rational rank of $G$}. 
Hence, $\rrk(G)=|S|$ for any maximal $\zz$-linearly independent subset $S$ in $G$.
In particular, $G$ is torsion if and only if $\rrk(G) = 0$.

The following is an extension of \cite[Lemma~3.4]{BL13}.

\begin{prop}\label{P:torsion-quotient-torsionfree-groups}\label{C:torsion-quotient-torsionfree-groups}
Let $G$ be a torsion-free abelian group and $H$ a subgroup of~$G$ with $r = \rrk(G/H) < \infty$.
Let $n \in \nat^+$.
Then $[G : nG] \leq n^r [H : nH]$.
Furthermore, if $G/H$ is finitely generated, then $[G : nG] = n^r [H : nH]$.
\end{prop}

\begin{proof}
For any $s\in\nat$ and elements $f_1,\dots,f_s\in G$ representing different classes in $G/nG$, denoting by $F$ the subgroup of $G$ generated by $H\cup\{f_1,\dots,f_s\}$, we have that $F/H$ is finitely generated and $f_1,\dots,f_s$ represent distinct classes in $F/nF$, whereby $s\leq [F:nF]$.
Therefore, to prove that $[G : nG] \leq n^r [H : nH]$,
we may assume without loss of generality that $G/H$ is finitely generated.

We choose $g_1,\dots,g_r\in G$ such that $(g_i+H\mid 1\leq i\leq r)$ 
is a maximal $\Z$-linearly independent family in $G/H$ and we denote by $G_0$ the subgroup of~$G$ generated by $H\cup\{g_1,\dots,g_r\}$.
The choice of $g_1,\dots,g_r$ implies that $G/G_0$ is torsion.

Assume that $G/H$ is finitely generated.
Then $G/G_0$ is finite and $G_0\simeq \Z^r\times H$, whereby
$[G_0 : nG_0] = n^r [H : nH]$.
Since $G$ is torsion-free, multiplication by $n$ induces an isomorphism $G/G_0\to nG/nG_0$. In particular $[G:G_0]=[nG:nG_0]$,
whereby $[G_0:nG_0]=[G:nG]$.
Therefore $[G:nG] = n^r [H:nH]$.
\end{proof}

By an \emph{ordered abelian group} we mean an abelian group $(G, +, 0)$ endowed with a total order relation $\leq$ such that,  for any $x, y, z \in G$ with $x \leq y$, one has $x + z \leq y + z$.
Note that any ordered abelian group is torsion-free as a group.
When $(G, +, 0, \leq)$ is an ordered abelian group, a \emph{convex subgroup} is a subgroup~$H$ of $G$ such that for all $x \in H$ and $y \in G$ with $0 \leq y \leq x$, we have $y \in H$.
One sees that the set of convex subgroups of $G$ is totally ordered by inclusion.
The cardinality of the set of proper convex subgroups of $G$ will be called the \emph{rank of~$G$} and denoted by $\rk(G)$.
By \cite[Prop.~3.4.1]{EP05}, one has $\rk(G) \leq \rrk(G)$.

Given a valuation $v$ on a field $K$, we refer to the rank of the value group $vK$ also as the \emph{rank of $v$} and denote it by $\rk(v)$.
Note that an ordered abelian group can be embedded into~$\rr$ if and only if it is of rank at most $1$. (Consequently, the rank-$1$ valuations on~$K$ correspond precisely to the non-trivial absolute values on~$K$.)

When $L/K$ is an extension of fields and $v$ and $w$ are valuations on $K$ and $L$ respectively, we say that $w$ \emph{extends} $v$ if $vK\subseteq wL$ and $v = w\vert_K$, and we call $(L, w)/(K, v)$ an \emph{extension of valued fields}.

\begin{prop}\label{P:val-fufi-ext-cases}
    Let $F/K$ be a function field in one variable and let $w$ be a valuation on $F$ and $v=w|_K$.
    Then one of the following holds:
    \begin{enumerate}[$(1)$]
        \item $Fw/Kv$ is a finite field extension and $wF/vK$ is a finitely generated $\zz$-module with $\rrk (wF/vK)=1$.
        \item $Fw/Kv$ is a function field in one variable and $wF/vK$ is finite.
        \item $Fw/Kv$ is an algebraic field extension and $wF/vK$ is a torsion group.
    \end{enumerate}
\end{prop}
\begin{proof}
    This follows from \cite[Theorem~3.4.3]{EP05}.
\end{proof}

\begin{lem}\label{L:hens-subfield-reduction}
    Let $(K,v)$ be a henselian valued field and $K_0$ a subfield of $K$.
    There exists a subfield $K_1$ of $K$ containing $K_0$ such that $vK_1/vK_0$ is torsion, $Kv/K_1v$ is algebraic and purely inseparable, $v|_{K_1}$ is henselian and the order of any non-trivial torsion element in $vK/vK_1$ is a power of $\car(Kv)$.
\end{lem}
\begin{proof}
     Using Zorn's Lemma, we choose $K_1$ as an extension of $K_0$ inside $K$ which is maximal with respect to the property that the quotient $vK_1/vK_0$ is torsion.
    
    Consider $t\in K$ algebraic over $K_1$. Then ${[vK_1(t):vK_1]\leq [K_1(t): K_1]<\infty}$,
    so $vK_1(t)/vK_1$ is torsion, and it follows by the choice of $K_1$ that $t\in K_1$.
    Therefore,~$K_1$ is relatively algebraically closed in $K$.
    
    Consider $t\in K$ transcendental over $K_1$. Then $K_1(t)/K_1$ is a rational function field in one variable,
    and hence it follows by \Cref{P:val-fufi-ext-cases} from the choice of~$K_1$ that the extension $K_1(t)v/K_1v$ is finite. 
    Having this for any $t\in K$, we conclude that the extension $Kv/K_1v$ is algebraic.

    Since $(K,v)$ is henselian and $K_1$ is relatively algebraically closed in $K$, it follows that $(K_1,v|_{K_1})$ is henselian as well.
    
    Consider $\xi\in Kv$ separable over $K_1v$. Since $(K,v)$ is henselian, $\xi$ is the residue of some element $t\in \mc{O}_v$ algebraic over $K_1$. 
    Since $K_1$ is algebraically closed in $K$, we obtain that $t\in K_1$, whereby $\xi\in K_1v$.
    This shows that $Kv/K_1v$ is purely inseparable.

    It remains to be shown that the order of any torsion element of $vK/vK_1$ is either $1$ or a power of $\car(Kv)$.
    Hence, for $q\in\nat^+$ not divisible by $\car(Kv)$ and $t\in\mg{K}$ such that $qv(t)\in vK_1$, we need to show that $v(t)\in vK_1$.
    Given any such~$q$ and $t$, we choose $y \in K_1$ such that $v(y) = qv(t)=v(t^q)$, whereby $yt^{-q}\in \mg{\mc{O}}_v$.
    Since $Kv/K_1v$ is purely inseparable, there exists a positive integer $m$ which is either $1$ or a power of $\car(Kv)$ such that the residue of $(yt^{-q})^m$ lies in $K_1v$.
    We fix $u\in\mg{K}_1$ such that $(yt^{-q})^m\in u+\mfm_{v}$.
    Since $(K, v)$ is henselian and $q \neq \car(Kv)$, one has $1 + \mf{m}_{v} \subseteq K^{\times q}$. Hence, there exists $z \in \mg{\mc{O}}_v$ with $(yt^{-q})^m = uz^q$. 
    Hence  $(t^mz)^q=u^{-1}y^m \in\mg{K}_1$.
    Since $K_1$ is relatively algebraically closed in $K$, it follows that $t^mz \in \mg{K}_1$.
    Since $v(z)=0$, we obtain that $mv(t) = v(t^mz) \in vK_1$. Since $qv(t)\in vK_1$ and $q$ is coprime to $m$, we conclude that $v(t)\in vK_1$, as desired.
\end{proof}

\section{Valued fields and residue forms} 
\label{S:DT}

We assume familiarity with standard facts from quadratic form theory over fields as covered in particular by the first chapters of \cite{Lam05} and \cite{EKM08}.
On this basis, we will give a succinct introduction to the theory of quadratic forms over valued fields with residue characteristic different from $2$.  
While in the case of a complete discrete valuation all these results are well-known, we have not found a presentation of this theory in the literature which would provide the statements in this generality.
Our presentation resembles that 
of \cite{Dur48}, a pioneer article on the topic of quadratic forms over valuation rings.
The terminology used in \cite{Dur48} is meanwhile a bit off the standard, and the statements there are formulated for a  valuation that is assumed to be complete of rank-$1$, while the arguments mostly work in more generality.
Our presentation highlights that large parts of the theory may be developed without recourse to a henselization or completion of the valuation ring.
To introduce the concepts and formulate a first few facts, we deal with quadratic forms over a local ring in which $2$ is invertible.

Let $R$ be a commutative ring, with unity.
A \emph{quadratic form over $R$} is a homogeneous polynomial $\varphi$ of degree $2$ in a prescribed finite tuple of variables whose length is called the  \emph{dimension of $\varphi$} and denoted by $\dim(\varphi)$.
Let $n\in\nat$ and consider an $n$-dimensional quadratic form $\varphi$ over $R$.
It gives rise to a map $R^n\to R$, defined by evaluating $\varphi$ on $n$-tuples,
as well as to an $R$-bilinear form 
$$ \mf{b}_\varphi : R^n \times R^n \to R : (x, y) \mapsto \varphi(x + y) - \varphi(x) - \varphi(y).$$
There is a unique symmetric $n\times n$-matrix $A$ over $R$ such that 
$\mf{b}_\varphi$ is given by $(x,y)\mapsto xAy^\tr$, and we call $\varphi$ \emph{nonsingular} if $\det(A)\in\mg{R}$.
We say that $\varphi$ \emph{represents $a\in R$} if $a=\varphi(x)$ for some $x\in R^n$.

Let $n\in\nat$ and $X=(X_1,\dots,X_n)$. 
Given elements $a_1,\dots,a_n\in R$, we denote the $n$-dimensional quadratic form $a_1X_1^2+\ldots+a_nX_n^2$ over $R$ by $\la a_1,\dots,a_n\ra_R$, or simply by $\la a_1,\dots,a_n\ra$ if the ring $R$ is clear from the context. A quadratic form of this shape (that is, where the  terms $X_iX_j$ with $i\neq j$ all have coefficient~$0$) is called \emph{diagonalized} or a \emph{diagonal form}.
Let
$\varphi,\psi\in R[X]$ be $n$-dimensional quadratic forms over $R$.
We call $\varphi$ and $\psi$ \emph{isometric} if $\psi(X)=\varphi(X\cdot C)$ for some invertible matrix $C\in\matr{n}(R)$, and we write $\varphi\simeq \psi$ to indicate this.
If $\varphi$ and $\psi$ are quadratic forms over $R$, $m=\dim(\varphi)$ and $n=\dim(\psi)$, then
we denote by $\varphi\perp\psi$ the $(m+n)$-dimensional quadratic form $\varphi(X_1,\dots,X_m)+\psi(X_{m+1},\dots,X_{m+n})$ over $R$.
We further denote by $\hh_R$ or by $\hh$ the quadratic form~$X_1X_2$ over $R$. 

The next few observations are under the assumption that $2\in\mg{R}$.
Note that this implies that $\hh_R\simeq\la 1,-1\ra$.
Consider an $n$-dimensional quadratic form $\varphi$ in $X=(X_1,\dots,X_n)$ over $R$.
Then $\varphi=XSX^{\tr}$ for a unique symmetric matrix $S\in\matr{n}(R)$.
We set $\det(\varphi)=\det(S)$ and call this the \emph{determinant of $\varphi$}.
Note that $\mf{b}_\varphi:R^n\times R^n\to R$ is given by $(x,y)\mapsto x(2S)y^\tr$, and hence $\varphi$ is nonsingular if and only if $\det(S)\in\mg{R}$, and we simply call $\varphi$ \emph{regular} in this case. (We avoid to use the term `regular' without assuming $2\in\mg{R}$, as it then occurs with different possible meanings in the literature.)
Note that $\det(\varphi\perp\psi)=\det(\varphi)\cdot \det(\psi)$ for any quadratic form $\psi$ over $R$.

While our main attention is to the case where $R$ is a valuation ring with $2\in\mg{R}$, the next few statements hold more generally when $R$ is local and $2\in\mg{R}$.

\begin{prop}\label{P:local-ring-chain-equivalence-step}
Assume that $R$ is local and $2\in\mg{R}$.
Let $\varphi$ be a quadratic form over $R$ and $b\in\mg{R}$.
Then $\varphi$ represents $b$ if and only if $\varphi\simeq \la b\ra_R\perp\psi$ for a quadratic form $\psi$ over $R$.
\end{prop}
\begin{proof}
Clearly, $b$ is represented by $\la b\ra_R$ and hence also by $\varphi$ if $\varphi\simeq \la b\ra_R\perp\psi$ for some quadratic form $\psi$ over $R$.
To prove the converse, assume that $\varphi(x) = b$ for some $x =(x_1,\dots,x_n)\in R^n$.
Since $b \in R^\times$ and $R$ is local, there exists $i \in \{ 1, \ldots, n \}$ with $x_i \in R^\times$. We may without loss of generality assume that  $x_1 \in R^\times$.
Let $(e_1,\ldots,e_n)$ be the canonical basis of $R^n$. We set $e_1' = x$ and further 
$e_i' = e_i - (2b)^{-1}\mf{b}_\varphi(e_1', e_i) e_1' \in R^n$ for $2\leq i\leq n$.
Then $(e_1',\dots,e_n')$ is also an $R$-basis of $R^n$.
Let $C\in\matr{n}(R)$ be the matrix of the base change from $(e_1', \ldots, e_n')$ to $(e_1, \ldots, e_n)$. 
Then $C$ is invertible in $\matr{n}(R)$.
Using that $\mf{b}_\varphi$ is $R$-bilinear, one computes that $\mf{b}_\varphi(e_1', e_i') = 0$ for $2\leq i\leq n$.
Thus, letting $X=(X_1,\dots,X_n)$, we have that 
$\varphi(X)\simeq \varphi(X\cdot C)=bX_1^2+\psi(X_2,\dots,X_n)$ for an $(n-1)$-dimensional quadratic form $\psi$ over $R$.    
\end{proof}

\begin{cor}\label{C:local-ring-chain-equivalence-step}
Assume that $R$ is local and $2\in\mg{R}$.
Let $\varphi$ be a quadratic form over $R$ and $n=\dim(\varphi)$.
Then $\varphi$ is regular if and only if $\varphi\simeq \la a_1,\dots,a_n\ra_R$ for certain $a_1,\dots,a_n\in\mg{R}$.
\end{cor}
\begin{proof}
    Clearly, if $\varphi\simeq \la a_1,\dots,a_n\ra_R$ for certain $a_1,\dots,a_n\in\mg{R}$, then $\varphi$ is regular.
    
    The converse is proven by induction on $n$.
    For $n=0$, the statement holds trivially.
    Assume that $n>0$.
Let $\mfm$ be the maximal ideal of $R$.
Then $\mg{R}=R\setminus \mfm$ and, since $\varphi$ is regular, the entries of the associated matrix cannot all lie in $\mfm$.
So there exist $x, y \in R^n$ with $\mf{b}_q(x,y)\notin \mfm$.
As 
$\mf{b}_q(x,y)=\varphi(x+y) - \varphi(x) - \varphi(y)$, we conclude that $\varphi(z)\notin \mfm$ 
for some $z\in\{x,y,x+y\}$.
Letting $a_1=\varphi(z)$, we have that $a_1\in R\setminus\mfm=\mg{R}$ and obtain by \Cref{P:local-ring-chain-equivalence-step} that $\varphi \simeq \la a_1 \ra \perp \psi$
for some $(n-1)$-dimensional quadratic form $\psi$ over $R$.
As $\varphi$ is regular, so is $\la a_1\ra\perp\psi$. Hence $a_1 \det(\psi) \in \mg{R}$, which implies that $\det(\psi)\in\mg{R}$.
Hence $\psi$ is regular.
By the induction hypothesis, we obtain that $\psi\simeq \la a_2,\dots,a_n\ra_R$ for certain $a_2,\dots,a_n\in\mg{R}$, whereby 
 $\varphi\simeq \la a_1,\dots,a_n\ra_R$.
\end{proof}

\begin{lem}\label{L:local-binary-rep}
Assume that $R$ is local with $2\in\mg{R}$.
Let $\mfm$ denote the maximal ideal of $R$. 
Let $x\in \mg{R}$ and $y,c\in R$.
If $|R/\mfm|=3$, then assume that $c\notin 1+\mfm$.
Then there exist $x'\in R$ and $y'\in \mg{R}$ such that $x^2-cy^2=x'^2-cy'^2$.
\end{lem}
\begin{proof}
If $y\in\mg{R}$, then we may take $x'=x$ and $y'=y$. 
Assume now that $y\notin \mg{R}$.
In view of the hypotheses, there exists $\lambda \in R$ such that $1-\lambda ^2c\in\mg{R}$ and $2\lambda x +(1+\lambda^2c)y\in\mg{R}$.
Set $x'=(1-\lambda ^2c)^{-1}((1+\lambda^2c)x+2\lambda c y)$ and $y'=(1-\lambda^2c)^{-1}(2\lambda x +(1+\lambda^2c)y)$. 
Then
\begin{eqnarray*}
(1-\lambda ^2c)^2(x'^2-cy'^2) & = & ((1+\lambda ^2c)x+2\lambda c y)^2-c(2\lambda x +(1+\lambda ^2c)y)^2\\ & = & (1-\lambda ^2c)^2(x^2-cy^2)\,,
\end{eqnarray*}
whereby $x^2-cy^2=x'^2-cy'^2$.
\end{proof}

For $n\in\nat$, a vector $x\in R^n$ is called \emph{primitive} if its coordinates generate $R$ as an ideal; as $R$ is local, this just says that one of the coordinates of $x$ lies in $\mg{R}$.

For valuation rings, the following statement is \cite[Theorem 4.5]{CLRR80}, where it is attributed to J.-L.~Colliot-Th\'el\`ene and M.~Kneser. We give a different proof.

\begin{prop}\label{L:local-rep}
Assume that $R$ is local with $2\in\mg{R}$.
Let $\varphi$ be a regular quadratic form over $R$, $n=\dim(\varphi )$ and $b\in R$. If $x_0^2b=\varphi(x)$ for some $x_0\in R$ and some primitive vector $x\in R^n$, then 
$b$ is represented by $\varphi$.
\end{prop}
\begin{proof}
Let $\mfm$ denote the maximal ideal of $R$.
By \Cref{C:local-ring-chain-equivalence-step}, we may assume without loss of generality that $\varphi =\la a_1,\dots,a_n\ra_R$ with $a_1,\dots,a_n\in\mg{R}$.
Let $x=(x_1,\dots,x_n)\in R^n$ be primitive and $x_0\in R$ such that $x_0^2b=\varphi(x)$.
We may permute $a_1,\dots,a_n$ and assume without loss of generality that, for some $m\in\{1,\dots,n\}$, we have 
 $x_1,\dots,x_m\in\mg{R}$ and $x_{m+1},\dots,x_n\in \mfm$.

Assume first that $|R/\mfm|>3$ or
$a_1\not\equiv b\bmod \mfm$.
Set $c=a_1^{-1}b$.
Using \Cref{L:local-binary-rep}, we obtain that 
$x_1^2-cx_0^2=z_1^2-cz_0^2$ for certain $z_0\in\mg{R}$ and $z_1\in R$. 
We set $y_1=z_0^{-1}z_1$ to obtain that 
 $$a_2x_2^2+\dots+a_nx_n^2=bx_0^2-a_1x_1^2=bz_0^2-a_1z_1^2=z_0^2(b-a_1y_1^2)\,.$$
Letting $y_i=z_0^{-1}x_i$ for $2\leq i\leq n$, we obtain that $b=a_1y_1^2+\ldots+a_ny_n^2$, as desired.

Consider now the case where $|R/\mfm|=3$ and $a_1\equiv b\bmod \mfm$.
If $x_0\in\mg{R}$, then
we may set $y_i=x_0^{-1}x_i$ for $1\leq i\leq n$
to obtain the desired equality.
Assume now that $x_0\in\mfm$.
Since $a_1x_1^2\in\mg{R}$ and $a_1x_1^2+\ldots+a_nx_n^2=bx_0^2\in\mfm$, we obtain that $m\geq 2$, whereby $x_2\in\mg{R}$.
If $a_2\not\equiv b\bmod \mfm$, then we switch the roles of $a_1$ and $a_2$ and conclude by the previous case.
So we are left to consider the case 
where $a_1\equiv a_2\equiv b\bmod \mfm$.
We set $c= a_1x_1^2 + a_2x_2^2$.
Then $c\equiv -b\not\equiv b\bmod \mfm$, and 
we conclude by the previous case that there exist $z, y_3, y_4, \ldots,y_n \in R$
such that $b = cz^2 + a_3y_3^2 + \ldots + a_ny_n^2$. Then letting $y_1 = zx_1$ and $y_2 = zx_2$, we obtain that $b=a_1y_1^2+\ldots+a_ny_n^2$.
\end{proof}

We collect some standard facts and definitions for quadratic forms over fields, referring to \cite[Chap.~II]{EKM08} for more details.
Let $n \in \nat$ and consider an $n$-dimensional quadratic form $\varphi$ over a field $K$.
We call $\varphi$ \emph{isotropic} if there exists $x \in K^n \setminus \{ 0 \}$ with $\varphi(x) = 0$, \emph{anisotropic} otherwise.
We have that $\varphi$ is nonsingular if and only if it is isometric to $(m \times \hh_K) \perp \psi$ for some natural number $m$ and a nonsingular anisotropic quadratic form $\psi$ over~$K$, where $m \times \hh_K$ denotes the $m$-fold orthogonal sum $\hh_K \perp \ldots \perp \hh_K$.
This $\psi$ is uniquely determined up to isometry and called the \emph{anisotropic part of $\varphi$}.
A nonsingular quadratic form whose anisotropic part is trivial ($0$-dimensional) is called \emph{hyperbolic}.
If $\car(K) \neq 2$, then every anisotropic quadratic form over $K$ is nonsingular, and by \Cref{C:local-ring-chain-equivalence-step} isometric to $\la a_1, \ldots, a_n \ra_K$ for some $a_1, \ldots, a_n \in \mg{K}$.

In the sequel, let $K$ be a field and $v$ a valuation on $K$.
Recall that we denote by~$\mc{O}_v$ the valuation ring of $v$, by~$\mfm_v$ its maximal ideal, by $Kv$ the residue field $\mc{O}_v/\mfm_v$ and by $vK$  the value group $v(\mg{K})$. We assume in the sequel that $\car(Kv)\neq 2$. This means that $2\in\mg{\mc{O}}_v$.
In particular, $\car(K)\neq 2$.

\begin{prop}\label{P:unimod-rep-valring}
    Let $n\in\nat$ and $a_1,\dots,a_n,b\in\mg{\mc{O}}_v$.
    \begin{enumerate}[$(a)$]
        \item If $\la a_1,\dots,a_n\ra_K$ represents $b$, then $n\geq 1$ and there exist $b_2,\dots,b_n\in\mg{\mc{O}}_v$ such that $\la a_1,\dots,a_n\ra_{\mc{O}_v}\simeq \la b,b_2,\dots,b_n\ra_{\mc{O}_v}$.
        \item If $\la a_1,\dots,a_n\ra_K$ is isotropic, then $n\geq 2$ and there exist $b_3,\dots,b_n\in\mg{\mc{O}}_v$ such that
        $\la a_1,\dots,a_n\ra_{\mc{O}_v}\simeq \hh_{\mc{O}_v}\perp\la b_3,\dots,b_n\ra_{\mc{O}_v}$.
    \end{enumerate}
\end{prop}

\begin{proof}
We set $\varphi=\la a_1,\dots,a_n\ra_{\mc{O}_v}$ and $\varphi_K=\la a_1,\dots,a_n\ra_K$.

    $(a)$ Assume that $\varphi_K$ represents $b$. Then $bx_0^2=a_1x_1^2+\ldots+a_nx_n^2$ for certain
 $x_0,\dots,x_n\in \mc{O}_v$, not all contained in $\mfm_v$.
 It follows that $x_i\in\mg{\mc{O}}_v$ for some $i\in\{1,\dots,n\}$.
We obtain by \Cref{L:local-rep} that  $b=a_1y_1^2+\ldots+a_ny_n^2$ for certain $y_1,\dots,y_n\in\mc{O}_v$.
We conclude by \Cref{P:local-ring-chain-equivalence-step} that $\varphi\simeq \la b\ra_{\mc{O}_v}\perp\psi$ for some $(n-1)$-dimensional quadratic form $\psi$ over $\mc{O}_v$.
Using \Cref{C:local-ring-chain-equivalence-step} we obtain that $\psi\simeq \la b_2,\dots,b_n\ra_{\mc{O}_v}$ for certain $b_2,\dots,b_n\in\mg{\mc{O}}_v$.
Then
 $\varphi\simeq \la b,b_2,\dots,b_n\ra_{\mc{O}_v}$.

$(b)$ Assume that $\varphi_K$ is isotropic. Hence $a_1x_1^2+\ldots+a_nx_n^2=0$ for certain $x_1,\dots,x_n\in K$ which are not all zero.
If $x_1=0$, then $\la a_2,\dots,a_n\ra_K$ is isotropic.
As every regular isotropic form over $K$ represents all elements of $K$, we obtain in any case that $-a_1$ is represented by $\la a_2,\dots,a_n\ra_K$.
Applying part $(a)$ to $\la a_2,\dots,a_n\ra_K$, we obtain that 
 $\la a_2,\dots,a_n\ra_{\mc{O}_v}\simeq \la -a_1,b_3,\dots,b_n\ra_{\mc{O}_v}$ for certain $b_3,\dots,b_n\in\mg{\mc{O}}_v$.
As $\la a_1,-a_1\ra_{\mc{O}_v}\simeq \hh_{\mc{O}_v}$, we conclude that $\varphi\simeq  \hh_{\mc{O}_v}\perp\la b_3,\dots,b_n\ra_{\mc{O}_v}$.
\end{proof}

A quadratic form $\varphi$ over $K$ is called \emph{$v$-unimodular} if $\varphi\simeq \la a_1,\dots,a_n\ra_K$ for $n=\dim(\varphi)$ and certain $a_1,\dots,a_n\in\mg{\mc{O}}_v$.

\begin{cor}\label{C:anis-part-unimod}
    A regular quadratic form over $K$ is $v$-unimodular if and only if its anisotropic part is $v$-unimodular.
\end{cor}
\begin{proof}
This follows by induction from \Cref{P:unimod-rep-valring}~$(b)$.
\end{proof}

\begin{prop}\label{P:nd-val-unimod-resform}
Let $n\in\nat$. Let $a_1,\dots,a_n,c_1,\dots,c_n\in\mg{\mc{O}}_v$. Assume that we have $\la a_1,\dots,a_n\ra_K\simeq \la c_1,\dots,c_n\ra_K$. Then
 $\la a_1,\dots,a_n\ra_{\mc{O}_v} \simeq \la c_1,\dots,c_n\ra_{\mc{O}_v}$ and
$\la \ovl{a}_1,\dots,\ovl{a}_n\ra_{Kv} \simeq \la \ovl{c}_1,\dots,\ovl{c}_n\ra_{Kv}$.
\end{prop}
\begin{proof}
    An isometry over $\mc{O}_v$ induces an isometry over $Kv$, so we only need to prove the first part. The proof is by induction on  $n$.
    For $n=0$ there is nothing to prove. 
    Assume that $n\geq 1$.
    As $c_1$ is represented by $\la a_1,\dots,a_n\ra_K$, 
    \Cref{P:unimod-rep-valring} yields that 
    $\la a_1,\dots,a_n\ra_{\mc{O}_v}\simeq \la c_1,b_2,\dots,b_n\ra_{\mc{O}_v}$ for certain $b_2,\dots,b_n\in\mg{\mc{O}}_v$.
    Then $\la c_1,\dots,c_n\ra_K\simeq \la a_1,\dots,a_n\ra_K\simeq \la c_1,b_2,\dots,b_n\ra_K$.
    Witt's Cancellation Theorem \cite[Chap.~I, Theorem 4.2]{Lam05} yields that 
    $\la c_2,\dots,c_n\ra_K\simeq \la b_2,\dots,b_n\ra_K$.
    Now, we obtain from the induction hypothesis, that 
    $\la c_2,\dots,c_n\ra_{\mc{O}_v}\simeq \la b_2,\dots,b_n\ra_{\mc{O}_v}$.
    Then $\la a_1,\dots,a_n\ra_{\mc{O}_v} \simeq \la c_1,b_2,\dots,b_n\ra_{\mc{O}_v}\simeq \la c_1,c_2\dots,c_n\ra_{\mc{O}_v}$.
\end{proof}

Assume that $\varphi$ is $v$-unimodular.
Let $n=\dim(\varphi)$ and $a_1,\dots,a_n\in\mg{\mc{O}}_v$ such that $\varphi\simeq \la a_1,\dots,a_n\ra_K$.
We associate to $\varphi$ the quadratic form $\ovl{\varphi}^v=\la \ovl{a}_1,\dots,\ovl{a}_n\ra_{Kv}$ over $Kv$. By \Cref{P:nd-val-unimod-resform}, the quadratic form $\ovl{\varphi}^v$ is uniquely determined up to isometry by $\varphi$. We call $\ovl{\varphi}^v$ the \emph{residue form of $\varphi$} (\emph{with respect to $v$}). We often simply write $\ovl{\varphi}$ for $\ovl{\varphi}^v$, if the context clarifies the reference to the valuation~$v$.

\begin{cor}\label{C:isotropy-down}
    Let $\varphi$ be a $v$-unimodular quadratic form over $K$. If $\varphi$ is isotropic, then $\ovl{\varphi}^v$ is isotropic.
\end{cor}
\begin{proof}
This follows directly from \Cref{P:unimod-rep-valring}~$(b)$.
\end{proof}

\begin{prop}
    Let $\varphi$ be a regular quadratic form over $K$.
    Let $\mc{C}$ be a set of representatives of $\mg{K}/v^{-1}(2vK)$.
    There exist $r\in\nat$, distinct $c_1,\dots,c_r\in \mc{C}$, 
    and $v$-unimodular quadratic forms $\varphi_1,\dots,\varphi_r$ over $K$ such that
    $\varphi\simeq c_1\varphi_1\perp\ldots\perp c_r\varphi_r$.
\end{prop}
\begin{proof}
    By the hypothesis on $\mc{C}$ we  have $vK=\bigcup_{c\in\mc{C}}(v(c)+2vK)$.
    We fix a diagonalization of $\varphi$ and then choose    
    $c_1,\dots,c_r\in\mc{C}$ such
    that 
    $\bigcup_{i=1}^r (v(c_i) + 2vK)$ contains the values of all entries of this diagonalization.
    For $i\in\{1,\dots,r\}$ and $a\in\mg{K}$ with $v(a)\in v(c_i) + 2vK$, we have $\la a\ra\simeq \la c_iu\ra=c_i\la u\ra$ for some $u \in\mg{\mc{O}}_v$.
    Hence, any of the diagonal entries can be replaced by $c_iu$ for some $i\in\{1,\dots,r\}$ and  $u\in\mg{\mc{O}}_v$.
    This leads to a presentation $\varphi\simeq c_1\varphi_1\perp\ldots\perp c_r\varphi_r$ where $\varphi_1,\dots,\varphi_r$ are $v$-unimodular quadratic forms over $K$.
\end{proof}

We come to the main statements of this section. 
They show that, if the valuation $v$ on $K$ is henselian, then unimodular quadratic forms over $K$ are characterised by their residue forms.
We will formulate those statements under the hypothesis that $1+\mfm_v\subseteq\sq{K}$.
This hypothesis holds in particular when $v$ is henselian. (It is equivalent to $v$ being \emph{$2$-henselian}; see \Cref{R:2hens}.)

\begin{thm}\label{P:hensval-res}
    Assume that $1+\mfm_v\subseteq\sq{K}\!$.
    Let  $\varphi$ be a $v$-unimodular quadratic form over $K$ and $b\in\mg{\mc{O}}_v$.
   The following equivalences hold:
    \begin{enumerate}[$(a)$]
        \item  $\varphi$ represents $b$  if and only if $\ovl{\varphi}^v$ represents $\ovl{b}$.
        \item  $\varphi$ is isotropic if and only if $\ovl{\varphi}^v$ is isotropic.
        \item $\varphi$ is hyperbolic if and only if $\ovl{\varphi}^v$ is hyperbolic.
    \end{enumerate}
\end{thm}

\begin{proof}
    Let $n\in\nat$ and $a_1,\dots,a_n\in\mg{\mc{O}}_v$ be such that 
    $\varphi\simeq \la a_1,\dots,a_n\ra_K$, whereby $\ovl{\varphi}\simeq \la \ovl{a}_1,\dots,\ovl{a}_n\ra_{Kv}$.
    
    $(a)$ If $\varphi$ represents $b$ over $K$, then it represents $b$ over $\mc{O}_v$, by \Cref{P:unimod-rep-valring}, and hence $\ovl{\varphi}^v$ represents $\ovl{b}$ over $Kv$.
    For the converse, assume that $\ovl{\varphi}^v$ represents~$\ovl{b}$ over $Kv$.
    Hence, there exist $x\in\mc{O}_v^n$ such that $\varphi(x)\equiv b\bmod \mfm_v$.
    It follows that $\varphi(x)b^{-1}\in 1+\mfm_v\subseteq\sq{K}$, whereby $\varphi(x)\sq{K}=b\sq{K}$. 
    Hence $\varphi$ represents $b$.

    $(b)$ 
    If $\varphi$ is isotropic, then so is $\ovl{\varphi}^v$, by \Cref{C:isotropy-down}.
    If $\ovl{\varphi}$ is isotropic, then
    $\ovl{\varphi}'=\la \ovl{a}_2,\dots,\ovl{a}_n\ra_{Kv}$ represents $-\ovl{a}_1$, and 
    it follows by $(a)$ that $\varphi'$ represents $-a_1$, whereby $\varphi$ is isotropic.
    
    $(c)$ We prove the equivalence by induction on $n$. If $n=0$, then $\varphi$ and~$\ovl{\varphi}^v$ are trivially hyperbolic.
    Assume that $n>0$.
    If $\varphi$ is anisotropic, then so is~$\ovl{\varphi}'$, by $(b)$, and as $n>0$, it follows that neither $\varphi$ nor $\ovl{\varphi}$ is hyperbolic.
    Assume that $\varphi$ is isotropic. By \Cref{P:unimod-rep-valring}~$(b)$, there exist $b_3,\dots,b_n\in\mg{\mc{O}}_v$ such that $\varphi\simeq \hh_K\perp \la b_3,\dots,b_n\ra_K$ and $\ovl{\varphi}\simeq \hh_{Kv}\perp\la\ovl{b}_3,\dots,\ovl{b}_n\ra_{Kv}$.
    Now, by Witt's Cancellation Theorem \cite[Chap.~I, Theorem 4.2]{Lam05}, $\varphi$ is hyperbolic if and only if $\rho=\la b_3,\dots,b_n\ra_K$ is hyperbolic, and $\ovl{\varphi}^v$ is hyperbolic if and only if $\ovl{\rho}^v=\la \ovl{b}_3,\dots,\ovl{b}_n\ra_{Kv}$ is hyperbolic.
    By the induction hypothesis, $\rho$ is hyperbolic if and only if $\ovl{\rho}^v$ is hyperbolic.
    From this together we conclude that $\varphi$ is hyperbolic if and only if $\ovl{\varphi}^v$ is hyperbolic.
\end{proof}

Let $\varphi$ be a quadratic form over $K$.
For $m\in\nat$, we denote by $m\times \varphi$ the $m$-fold orthogonal sum $\varphi\perp\ldots\perp\varphi$.
We say that $\varphi$ is \emph{torsion} if $m\times \varphi$ is hyperbolic for some $m\in\nat^+$.

Part of the following statement recovers \cite[Theorem 1]{Dur48} and its corollary.

\begin{thm}[Durfee]\label{T:Durfee} 
    Assume that $1+\mfm_v\subseteq\sq{K}$.
    Let $\varphi$ be a quadratic form over $K$. 
    Let $r\in\nat$, $c_1,\dots,c_r\in \mg{K}$ with $v(c_i)\not\equiv v(c_j)\bmod 2vK$ for $1\leq i<j\leq r$ and let $\varphi_1,\dots,\varphi_r$ be 
    $v$-unimodular quadratic forms over $K$ such that we have 
     $\varphi\simeq c_1\varphi_1\perp\ldots\perp c_r\varphi_r$.
    Then $\varphi$ is anisotropic, hyperbolic or torsion if and only if 
    all residue forms $\ovl{\varphi}_1^v,\dots,\ovl{\varphi}_r^v$ have the corresponding property.
\end{thm}

\begin{proof}
    Suppose that $\ovl{\varphi}_i^v$ is isotropic for some $i\in\{1,\dots,r\}$.
    Then $\varphi_i$ is isotropic, by \Cref{P:hensval-res}~$(b)$, and hence $\varphi$ is isotropic.
    
    Suppose now that $\ovl{\varphi}_i^v$ is anisotropic for $1\leq i\leq r$.
    Then $\varphi_1,\dots,\varphi_r$ are anisotropic and represent only elements 
    of $\mg{K}$ that lie in $v^{-1}(2vK)$.
    Consider $x\in K^n\setminus\{0\}$ for $n=\dim(\varphi)$ and let $a=\varphi(x)$.
    As $\varphi_1,\dots,\varphi_r$ are  anisotropic and $\varphi\simeq c_1\varphi_1\perp\ldots\perp c_r\varphi_r$, we obtain that $a=c_1a_1+\ldots+c_ra_r$ for some $(a_1,\dots,a_r)\in K^r\setminus\{0\}$ with $v(a_1),\dots,v(a_r)\in 2vK\cup\{\infty\}$.
    For $1\leq i<j\leq r$, we have $v(c_i)\not\equiv v(c_j)\mod 2vK$,  so either $v(c_ia_i)\neq v(c_ja_j)$ or $a_i=a_j=0$. As $(a_1,\dots,a_r)\neq 0$, we conclude that 
    $v(a)=\min\{v(a_ic_i)\mid 1\leq i\leq r\}<\infty$, whereby $a\neq 0$.
    This argument shows that $\varphi$ is anisotropic.

Having thus shown the equivalence for anisotropy, we will use it now to establish the equivalence for hyperbolicity.
For $i\in\{1,\dots,r\}$, let $\psi_i$ denote the anisotropic part of $\varphi_i$ and let $m_i\in\nat$ be such that $\varphi_i\simeq m_i\times\hh\perp\psi_i$.
Set $m=m_1+\ldots+m_r$ and $\psi=c_1\psi_1\perp\ldots\perp c_r\psi_r$.
Then $\varphi\simeq m\times \hh\perp \psi$.
By \Cref{C:anis-part-unimod}, $\psi_1,\dots,\psi_r$ are $v$-unimodular.
It follows by the above that $\psi$ is anisotropic.
Therefore $\varphi$ is hyperbolic if and only if $\psi$ is trivial, if and only if 
$\psi_1,\dots,\psi_r$ are all trivial.
Furthermore, for $1\leq i\leq r$, we have by part $(c)$ of \Cref{P:hensval-res} that $\ovl{\varphi}_i^v$ is hyperbolic if and only if $\varphi_i$ is hyperbolic, which is if and only if the anisotropic part $\psi_i$ is trivial.
This establishes the equivalence for hyperbolicity.

Consider now $m\in\nat$.
Note that $m\times \varphi\simeq c_1(m\times \varphi_1)\perp\ldots\perp c_r(m\times \varphi_r)$ and 
that $m\times \varphi_i$ is $v$-unimodular for $1\leq i\leq r$.
Having this for all positive integers~$m$, the equivalence for being torsion follows from that for hyperbolicity.
\end{proof}

Preceding treatments of quadratic forms over complete nondyadically discretely valued (or more generally $2$-henselian nondyadically valued) fields often focussed more on describing the structure of the Witt ring of quadratic forms than on determining isotropy etc.~ of individual quadratic forms. While we do not use the Witt ring in the present article,  we wish to mention the following direct consequence of \Cref{T:Durfee}, which relates to the Witt ring point of view.

Two regular quadratic forms $\varphi$ and $\psi$ over a field of characteristic different from $2$ are called \emph{Witt equivalent} if their
anisotropic parts are isometric, which is if and only if
 $\varphi \perp -\psi$ is hyperbolic.

\begin{cor}
Assume that $1+\mfm_v\subseteq\sq{K}$.
    Let $r\in\nat$, $c_1,\dots,c_r\in \mg{K}$ with $v(c_i)\not\equiv v(c_j)\bmod 2vK$ for $1\leq i<j\leq n$ and let $\vartheta_1,\psi_1,\dots,\vartheta_r,\psi_r$ be 
    $v$-unimodular quadratic forms over $K$. 
    Then the forms $c_1\vartheta_1\perp\ldots\perp c_r\vartheta_r$ and $ c_1\psi_1\perp\ldots\perp c_r\psi_r$  over $K$ are Witt equivalent if and only if, for $1\leq i\leq r$, the forms $\ovl{\vartheta}_i$ and $\ovl{\psi}_i$ over $Kv$ are Witt equivalent.
\end{cor}
\begin{proof}
    Letting $\varphi_i=\vartheta_i\perp-\psi_i$ for $1\leq i\leq r$, the statement follows from the criterion for the hyperbolicity of the quadratic form  $\varphi=c_1\varphi_1\perp\ldots\perp c_r\varphi_r$ over~$K$ which is contained in \Cref{T:Durfee}.
\end{proof}

\begin{rem}\label{rem:Durfee}
When restricting attention to complete discrete valuation rings and their fraction fields, most (if not all) of the results in this section are well-known.
To the best of our knowledge these first appeared explicitly in Durfee's doctoral research \cite{Dur48} for fields with a complete rank-$1$ valuation.
Inspection reveals that Durfee's proofs (and other versions found later) mostly work without changes for general valued fields $(K, v)$ with $1 + \mf{m}_v \subseteq K^{\times 2}$ and $2 \in \mc{O}_v^\times$, thus in particular for non-dyadic henselian valued fields, which would be sufficient for the purposes in the rest of this article.
For our own and the reader's convenience, we have included here an elementary and self-contained exposition which postpones adding the assumption $1 + \mf{m}_v \subseteq K^{\times 2}$ until the very end.
\end{rem}

\section{The generalised $u$-invariant} 
\label{S:genu}

By the Artin-Schreier Theorem, a field $K$ admits a field ordering if and only if~$-1$ is not a sum of squares in $K$; see e.g.~\cite[Chap.~VIII, Theorem 1.10]{Lam05}.
This motivates to call the field $K$ \emph{real} in this case, and \emph{nonreal} otherwise.
Given a field ordering, a quadratic form is assigned a corresponding signature, and there is a notion of being positive or negative definite at this ordering. 
A quadratic form which is positive definite at some ordering cannot be universal, since it cannot represent $-1$. Therefore the  definition of the $u$-invariant formulated in the introduction for nonreal fields would have value $\infty$ for any real field.
Elman and Lam \cite{EL73} suggested a variation to the definition of the $u$-invariant that is meaningful also for real fields while not changing its meaning for nonreal fields.
We refer to \cite[Chap.~8]{Pfi95} for a general discussion of the $u$-invariant, including several variations on how to define it for real fields on the one hand and for fields of characteristic $2$ on the other hand.

In this section, we discuss relations between the values of the $u$-invariant for a given valued field and its residue field. 
There seems a want in the literature for such a treatment covering both real fields and fields of characteristic $2$ in a compatible way.
Most of the statements here are well-known in the case where the valued field is nonreal and the residue field is of characteristic different from~$2$. 
Including the case of real fields in the discussion leads to some technicalities, which can be skipped if one restricts the attention to nonreal fields.
\medskip

Let $K$ be a field.
Recall that, if $\car(K) \neq 2$, we defined a quadratic form~$\varphi$ over $K$ to be a \emph{torsion form} if $m\times\varphi=\varphi\perp\ldots\perp\varphi$ is hyperbolic for some $m\in\nat^+$.
When $\car(K)=2$, we take the convention that every quadratic form over $K$ is torsion. 
The \emph{general $u$-invariant of $K$} as defined by Elman and Lam in~\cite{EL73} is given by 
$$u(K) \, = \, \sup\{\dim(\varphi)\mid \varphi\mbox{ anisotropic torsion form over }K\}\,\in \nat\cup\{\infty\}.$$
Note that, if $K$ is nonreal, then every anisotropic quadratic form over $K$ is a torsion form, and hence in this case the definition can be simplified to 
$$u(K) \, = \, \sup\{\dim(\varphi)\mid \varphi\mbox{ anisotropic quadratic form over }K\},$$
which corresponds to how the $u$-invariant was originally defined by Lenz. 
This further agrees with the definition mentioned in the introduction when $K$ is nonreal of characteristic different from $2$.

Pfister's Local-Global Principle \cite[Chap.~VIII, Sect.~3]{Lam05} states that, over a real field, a regular quadratic form is torsion if and only if it has signature zero at every field ordering.  Therefore, if~$K$ is real, then the $u$-invariant of $K$ has to be even or infinite. 

The following construction will be used to produce certain torsion forms in characteristic different from $2$.

\begin{lem}\label{L:torsion}
        Assume that $\car(K)\neq 2$. 
        Let $\vartheta$ be a regular quadratic form over~$K$.
        Let $c\in\mg{K}$ and $k\in\nat$ be such that $c$ is a sum of $2^k$ squares in $K$.
        Then $\vartheta\perp -c\vartheta$ is a torsion form. More precisely, $2^k\times (\vartheta\perp -c\vartheta)$ is hyperbolic. 
\end{lem}
\begin{proof}
    The hypothesis implies that the form $2^k \times \langle 1, -c \rangle_K$ is isotropic, and since it is a Pfister form, it follows by \cite[Theorem X.2.9]{Lam05} that it is hyperbolic.
    Consequently, $2^k\times \la a,-ac\ra$ is hyperbolic for every $a\in\mg{K}$.
    Let $n=\dim(\vartheta)$.
    By \Cref{C:local-ring-chain-equivalence-step}, there exist $a_1,\dots,a_n\in\mg{K}$ such that $\vartheta\simeq \la a_1,\dots,a_n\ra_K$.
    Then $2^k\times(\vartheta\perp-c\vartheta)\simeq \bigperp_{i=1}^n (2^k\times \la a_i,-ca_i\ra_K)$, which is hyperbolic.
\end{proof}

Let $v$ in the sequel be a valuation on $K$.

\begin{lem}\label{lift}
    Let $n\in\nat$ and let $\psi$ be an $n$-dimensional quadratic form over~$Kv$.
    Then there exists an $n$-dimensional quadratic form $\Psi$ over $\mc{O}_v$ which reduces modulo $\mfm_v$ to $\psi$.
    Moreover, if $\psi$ is anisotropic, then any such form $\Psi$ is anisotropic over $K$ and satisfies $v(\Psi(x))\in 2vK$ for every $x\in K^n\setminus\{0\}$.
\end{lem}
\begin{proof}
    See \cite[Proposition 4.2]{BDD23}.
\end{proof}

In some cases where we will use \Cref{lift}, we want the lifted form $\Psi$ in addition to be a torsion form over $K$. 
To achieve this some extra considerations are necessary when $K$ is real.

For a ring $R$ we denote by $\sos{R}$ the set of elements of $R$ which are sums of squares in $R$.

\begin{lem}\label{torsbinlift}
Assume that $Kv$ is nonreal and $\car(K)\neq 2$. 
The following hold:
\begin{enumerate}[$(a)$]
    \item
Every element of $Kv$ is the residue modulo $\mfm_v$ of an element of $\mc{O}_v\cap \sos{K}$.
\item Every anisotropic even-dimensional  quadratic form over $Kv$ is the reduction modulo $\mfm_v$ of a 
quadratic form defined over $\mc{O}_v$ which is an anisotropic  
torsion form over $K$.
\end{enumerate}
\end{lem}
\begin{proof} 
$(a)$  
If $\car(Kv)=2$, then for $x\in\mc{O}_v$ we have $x+2(x^2+1)\equiv x\bmod \mfm_v$  and $x+2(x^2+1)=2(x+\frac{1}4)^2+\frac{15}8\in \sos{K}$.
If $Kv$ is nonreal with $\car(Kv)\neq 2$, then there exists $z\in\sos{\mc{O}}$ with $z\equiv -1\bmod\mfm_v$, and 
for $x\in\mc{O}_v$ we obtain that $\frac{1}4((x+1)^2+z(x-1)^2)\in\sos{\mc{O}_v}$ and
$x\equiv \frac{1}4((x+1)^2+z(x-1)^2)\bmod\mfm_v$.

$(b)$
Consider an arbitrary anisotropic $2$-dimensional form $\beta$ over $Kv$.
By $(a)$, we may choose $a,b,c\in \mc{O}_v\cap \sos{K}$ such that
the $2$-dimensional quadratic form $B=aX_1^2+bX_1X_2-cX_2^2$ over $\mc{O}_v$ reduces modulo $\mfm_v$ to $\beta$.
Since $\beta$ is anisotropic over $Kv$, so is $B$ over $K$.
Note that
$B\simeq a\la 1,-(b^2+4ac)\ra$ over $K$ and $b^2+4ac\in\sos{K}$, whence it follows by \Cref{L:torsion} that $B$ is a torsion form.
This shows the claim for $2$-dimensional quadratic forms.
By \cite[Propositions 7.29 and 7.31]{EKM08} 
 every even-dimensional quadratic form is an orthogonal sum of $2$-dimensional quadratic forms. Hence the statement follows in general.
\end{proof}

The following is a partial extension of \cite[Theorem 5.2]{BL13}, including now the case of residue characteristic $2$.

\begin{prop}\label{P:BL}
We have 
$$u(K) \geq [vK : 2vK]\cdot 2\lfloor \mbox{$\frac{u(Kv)}{2}$}\rfloor\,.$$
Moreover, $$u(K)\geq [vK:2vK]\cdot u(Kv)$$ 
holds under any of the following conditions:
$(i)$\, $K$ is nonreal; 
$(ii)$\, $Kv$ is real;
$(iii)$\,~$u(Kv)$ is even;
$(iv)$ $v(\sos{K}\setminus\{0\})\not\subseteq 2vK$.
\end{prop}
\begin{proof}
    In the case where $Kv$ is real, the statement is fully covered by \cite[Theorem~5.2]{BL13}. 
    Hence we only consider the complementary situation and assume from now on that $Kv$ is nonreal.
    
    Let $n,r\in\nat$ arbitrary with $2^r\leq [vK:2vK]$ and $n\leq u(Kv)$, and assume that~$n$ is even in the case where $K$ is real and $v(\sos{K}\setminus\{0\})\subseteq 2vK$. 
    We will show that $u(K)\geq 2^rn$.
    This will establish the statement.
    
    By the choice of $n$ there exists an anisotropic $n$-dimensional quadratic form~$\varphi$ over $Kv$.
    Then $\varphi$ is the reduction modulo $\mfm_v$ of 
    an $n$-dimensional quadratic form~$\Phi$ over $\mc{O}_v$, and by \Cref{torsbinlift} we may choose $\Phi$ to be a torsion form over~$K$ when $n$ is even.
    Also if $K$ is nonreal, $\Phi$ is torsion.
    In any case, since $\varphi$ is anisotropic, it follows by \Cref{lift} that $v(\Phi(x))\in 2vK$ for every $x\in K^n\setminus\{0\}$.

    We may assume that $r\geq 1$ unless $vK=2vK$.
    By the choice of $r$, we may choose $c_1, \ldots, c_r \in \mg{K}$ such that $v(c_1),\dots,v(c_r)$ represent $\ff_2$-linearly independent classes in $vK/2vK$, and we may choose to have $c_1\in -\sos{K}$ except when $v(\sos{K}\setminus\{0\})\subseteq 2vK$. 
    We denote by $\mc{P}$ the power set of $\{1,\dots,r\}$, and for $I\in\mc{P}$ we set $c_I=\prod_{i\in I}c_i$.
    We consider the quadratic form $\Psi = \bigperp_{I\in\mc{P}} c_I\Phi$ over~$K$.
    It follows by the properties of $\Phi$ and $c_1,\dots,c_r$ that $\Psi$ is anisotropic.
    We claim that $\Psi$ is torsion.
    This is obvious when $\Phi$ is torsion. 
    Assume that $\Phi$ is not torsion. In this case  $n$ is odd, $vK\neq 2vK$, $r\geq 1$ and~$-c_1\in\sos{K}$, and since $\Psi=\Theta\perp c_1\Theta$ with $\Theta= \bigperp_{I\in\mc{P'}} c_I\Phi$ where $\mc{P}'$ is the power set of $\{2,\dots,r\}$, we conclude by \Cref{L:torsion} that $\Psi$ is torsion.
    Hence $u(K)\geq \dim(\Psi) = 2^r n$.
    Having this for all $n$ and $r$ as above, the statement follows.
\end{proof}

We obtain the following lower bound for  the growth of the $u$-invariant under finitely generated field extensions in terms of the transcendence degree. 
The argument is well-known at least for diagonal quadratic forms over nonreal fields, see~e.g.~\cite[Cor.~2~$(ii)$]{Pum09}.

\begin{cor}\label{fg-u-lowerbound}
    Let $F/K$ be a finitely generated field extension and let $n$ be its transcendence degree. 
    Then there exists a finite field extension $L/K$ such that $u(F)\geq 2^nu(L)$.
\end{cor}
\begin{proof}
    The statement is trivial when $n=0$.
    Consider the case where $n=1$. 
    In this case $F/K$ is a function field in one variable.
    We pick any discrete valuation~$v$ on $F$ which is trivial on $K$.
    We set $L=Fv$.
    Then $L/F$ is a finite field extension and $[\zz:2\zz]=2$.
    Moreover, if $F$ is real but $Fv$ is nonreal, then by \cite[Proposition~4.2]{BGVG14}, there exists $x\in \sos{F}$ with $v(x)=1$.
    Hence, in any case we obtain by \Cref{P:BL} that $u(F)\geq 2u(L)$. 
    This proves the statement for $n=1$. The general statement follows by induction on $n$.
\end{proof}

Determining the $u$-invariant of a valued field in terms of the value group and the $u$-invariant of the residue field is possible in residue characteristic different from $2$ provided that roots of quadratic equations can be lifted from the residue field. 
This is also well-known over nonreal fields, see e.g.~\cite[Theorem 4]{Pum09}.

\begin{prop}\label{P:u-hens}
Assume that $1+\mfm_v\subseteq\sq{K}$ and $\car(Kv)\neq 2$. 
Then $$u(K)= [vK:2vK]\cdot u(Kv)\,.$$
\end{prop}
\begin{proof}    
In view of \Cref{P:hensval-res}~$(a)$, the hypothesis implies that $K$ is real if and only if $Kv$ is real.
    In particular, \Cref{P:BL} yields that 
    $$u(K)\geq [vK:2vK]\cdot u(Kv)\,.$$
    It remains to show the converse inequality.

    For this, we set $s=[vK:2vK]$ and $u=u(Kv)$ and assume without loss of generality that $s<\infty$.
    Note that $v^{-1}(2K)=\mg{\mc{O}}_v\sq{K}$. 
    We fix representatives $c_1,\dots,c_s\in\mg{K}$ of the different classes of $\mg{K}/\mg{\mc{O}}_v\sq{K}$.
    Let $\varphi$ now be an anisotropic torsion form over $K$.
    We write $\varphi\simeq c_1 \varphi_1 \perp \ldots \perp c_s\varphi_s$ where $\varphi_1,\dots,\varphi_s$ are $v$-unimodular forms over $K$.
    Then $\ovl{\varphi}_1^v,\dots,\ovl{\varphi}_s^v$ are anisotropic torsion forms over $K$, by \Cref{T:Durfee}.
    Hence $\dim(\varphi_i)=\dim(\ovl{\varphi}_i^v)\leq u$ for $1\leq i\leq s$, whereby $\dim(\varphi)=\sum_{i=1}^s\dim(\varphi_i)\leq s\cdot u$.
    This shows that $u(K)\leq s\cdot u$.
\end{proof}

A similarly general statement covering the case of residue characteristic $2$ seems out of reach.
Only the following statement is known to hold without restriction on the residue characteristic.
One may trace it back to Kaplansky \cite{Kap53}, who observes that $u(K)=2u(k)$ in the case where $K=k(\!(t)\!)$ for a field $k$ with $\car(k)\neq 2$.

If $\car(K)=2$, we denote by $[K:K^{(2)}]$ the degree of imperfection of $K$, i.e.~the degree of the field extension $K/K^{(2)}$ where $K^{(2)}=\{x^2\mid x\in K\}$.

\begin{thm}[Kaplansky; Mamone-Moresi-Wadsworth]\label{T:MMW}
    Assume that $v$ is discrete and  henselian.
    If $\car(K)=2$, then assume further that $[K:K^{(2)}]=2[Kv:Kv^{(2)}]$.
    Then $u(K)=2u(Kv)$.
\end{thm}
\begin{proof}
    When $\car(Kv)\neq 2$ this follows from \Cref{P:u-hens}.
    See \cite[Theorem 2]{MMW91} for the case when $\car(Kv)=2$.
\end{proof}

Note that, if $\car(K)=2$ and $v$ is complete and discrete, then $[K:K^{(2)}]=2[Kv:Kv^{(2)}]$ holds, whereby \Cref{T:MMW} applies.
In a more general context, when $\car(Kv)=2$, one may have an inequality $u(K) > [vK : 2vK] \cdot u(Kv)$, even when $(K, v)$ is henselian and complete.
We sketch an example, where we omit the verification of valuation-theoretic details.
\begin{ex}
Let $\ovl{\ff}_2$ denote the algebraic closure of $\ff_2$ and $K$ the perfect closure (maximal purely inseparable extension) of $\ovl{\ff}_2(t)$ or of $\ovl{\ff}_2(\!(t)\!)$. That is, $K = \overline{\ff}_2(t^{2^{-n}}\vert\, n\in\nat)$ or $K=\overline{\ff}_2(\!(t)\!)(t^{2^{-n}}\vert \,n\in\nat)$, respectively. 
Let $v$ denote the unique extension of the $t$-adic valuation on $\ovl{\ff}_2(t)$ and $\ovl{\ff}_2(\!(t)\!)$ to $K$.
We also denote the natural extension of this valuation to the completion $K^v$ by $v$.
Note that $(K^v, v)$ is a complete rank-$1$ valued field. In particular, it is henselian.
Since $K$ is perfect, so is $K^v$, whereby trivially $1 + \mf{m}_{v} \subseteq \sq{(K^v)}$.
The value group $vK$ is $2$-divisible, and the residue field $Kv = \overline{\ff}_2$ is algebraically closed, whereby $[vK : 2vK] = 1 = u(Kv)$.
But $u(K) > 1$, because the quadratic form $X_1^2 + X_1X_2 + t^{-1}X_2^2$ is anisotropic over $K^v$.

With a bit more work, an example can also be constructed of a complete henselian valued field $(K, v)$ with $\car(K) = 0$, $\car(Kv) = 2$, and $1 + \mf{m}_v \subseteq \sq{K}$, whereas $u(K) > [vK : 2vK] \cdot u(Kv)$.
\end{ex}

\section{The generalised strong $u$-invariant} 
\label{S:genu-strong}

We are now prepared to study the $u$-invariant of  function fields in one variable.
Let $K$ be a field. 
Extending \cite[Sect.~5]{BGVG14} (by assuming nothing further on $K$),
we define 
$$\su(K)\,=\, \half\sup\,\{u(F)\mid F/K \mbox{ function field in one variable}\}\,\in \mbox{$\frac{1}{2} $}\nat\cup\{\infty\}.$$
We call $\su(K)$ the \emph{generalised strong $u$-invariant of $K$}.
For nonreal fields of characteristic different from $2$, $\su$ coincides with the `strong $u$-invariant' defined in ~\cite{HHK09}.

\begin{prop}\label{P:u-rafufi}
    Let $L/K$ be a finite field extension.
    Then $$u(L)\leq \hbox{$\frac{1}2$}u(L(X))\leq \su(K)\,.$$
    If $L/K$ is a simple extension or $\car(K)\neq 2$, then
    $u(L)\leq \hbox{$\frac{1}2$}u(K(X))\leq \su(K)$.
\end{prop}
\begin{proof}
    Since $L$ is the residue field of a discrete valuation on $L(X)$ and since $L(X)$ is real if and only if $L$ is real, the first inequality follows from \Cref{P:BL}.
    The second inequality is clear from the definition of $\su(K)$. 
    The refinement under the assumption that $L/K$ is simple or $\car(K)\neq 2$ now follows by \cite[Proposition~6.1]{BL13} and its proof.
\end{proof}

\begin{ex}\label{EX:CM}
    In \cite{CTM04}, an example is constructed of a field $K$ such that $u(L)=2$ holds for every finite field extension $L/K$, while there exists a pair of quadratic forms in $5$ variables over $K$ having no nontrivial common zero. As explained in \cite[Introduction]{Bec18}, this implies that $u(K(X))\geq 6$.
    Hence in this case we obtain that $\su(K)\geq \frac{1}2u(K(X))\geq 3>2=\sup\{u(L)\mid L/K\mbox{ finite field extension}\}$.
\end{ex}

\begin{prop}\label{exact-uinv-ffiov}
    Assume that $u(L) = \su(K)$ for every finite field extension $L/K(\sqrt{-1})$.
    Then $u(F) = 2\,\su(K)$ for every function field in one variable $F/K$.
\end{prop}
\begin{proof}
Consider a function field in one variable $F/K$.
By the definition of $\su(K)$, we have $u(F) \leq 2\su(K)$.
We fix $x\in F$ transcendental over $K$.
There exists a discrete valuation $w$ on $F$ which is trivial on $K$ and with $w(x^2+1)>0$. Then $Fw$ is a finite field extension of $K(\sqrt{-1})$, whence $u(Fw) = \su(K)$ by the hypothesis.
Since $[wF : 2wF] = 2$, we conclude by \Cref{P:BL} that $u(F) \geq 2\su(K)$. Therefore $u(F) =2\su(K)$.
\end{proof}

The following is a partial extension of \cite[Theorem]{EW87} with no assumption on $\car(K)$. It is also a variation of \cite[Theorem 7.3]{BDD23}.
\begin{prop}\label{P:herquacl}
The following are equivalent:
\begin{enumerate}[$(i)$]
    \item\label{it:herquacl1} $K(\sqrt{-1})$ has no finite field extension of even degree,
    \item\label{it:herquacl2} $u(F)=2$ for every function field in one variable $F/K$,
    \item\label{it:herquacl3} $u(F)< 4$ for some function field in one variable $F/K$.
\end{enumerate}
If these equivalent conditions hold, then $\su(K)=1$, otherwise $\su(K)\geq 2$.
\end{prop}
\begin{proof}
Since any nontrivial torsion form over a field $F$ is indefinite at every field ordering of $F$, the implication 
$({\it\ref{it:herquacl1} \Rightarrow \ref{it:herquacl2}})$ follows from \cite[Theorem~7.3]{BDD23}.
Since 
 $({\it\ref{it:herquacl2} \Rightarrow \ref{it:herquacl3}})$ is obvious, it therefore remains to prove $({\it\ref{it:herquacl3}\Rightarrow\ref{it:herquacl1}})$.
 
Let $F/K$ be an arbitrary function field in one variable.
Assume that $K(\sqrt{-1})$ has a finite extension of even degree.
Then the proof of \cite[Theorem 7.3]{BDD23} shows that there exists an anisotropic totally indefinite $2$-fold Pfister form $\pi$ over $F$.
Then $\pi$ has signature zero at every ordering of $F$.
Using Pfister's Local-Global Principle \cite[Chap. VIII, Theorem 6.14]{Lam05} if $K$ is real, it follows that $\pi$ is a torsion form.
We conclude that $u(F)\geq \dim(\pi) \geq 4$.
\end{proof}

Note that by its definition, $\su(K)$ is always in $\frac{1}2\nat\cup\{\infty\}$, but we do not know whether it is always in $\nat$. 
On the other hand, we do not know of any example where $\su(K)$ is not a power of $2$.

In the following example and the subsequent statement, we use the concept of $\mc{C}_i$-fields (where $i\in\nat$) introduced by S.~Lang in \cite{Lan52}.

\begin{ex}\label{Ex;Ci}
    Let $i\in\nat$ be such that $K$ is a $\mc{C}_i$-field. Then $u(K)\leq 2^i$.
    Moreover, for any $j\in\nat$ and any field extension $L/K$ of transcendence degree $j$, \cite[Theorem 2a]{Nag57} (which is \cite[Theorem~6]{Lan52} without the hypothesis on `normic forms') yields that $L$ is a $\mc{C}_{i+j}$-field, whereby $u(L)\leq 2^{i+j}$.
    Applying this with $j=1$, we obtain that $\su(K)\leq 2^i$.
\end{ex}

\begin{thm}[Lang]\label{T:Lang}
Let $(K,v)$ be a henselian discretely valued field with $\car(K)=0$ and algebraically closed residue field $Kv$.
Then $K$ is a $\mc{C}_1$-field. 
In particular, for any $i\in\nat$, we have $u(F)\leq 2^{i+1}$ for any field extension $F/K$ of transcendence degree $i$. 
Furthermore $\su(K) = 2$.
\end{thm}
\begin{proof}
    That $K$ is a $\mc{C}_1$-field is proven in \cite[Theorem 10]{Lan52} in the case where $(K,v)$ is complete.
    Let $K^v$ denote the completion of $K$ with respect to $v$.
    By \cite[Theorem 5.9]{Kuh16}, the hypotheses on $(K,v)$ imply that $K$ is existentially closed in~$K^v$.
    In other terms, every multivariate polynomial over $K$ has a zero over $K$ provided that it has a zero over $K^v$. 
    Since $K^v$ is a $\mc{C}_1$-field, we conclude that~$K$ is a $\mc{C}_1$-field.

    As $K$ carries a discrete valuation, $K(\sqrt{-1})$ has a quadratic field extension, so \Cref{P:herquacl} yields that $\su(K)\geq 2$.
    The rest follows by \Cref{Ex;Ci}.
\end{proof}

\begin{prop}\label{P:val-su-inequality}
    Let $v$ be a valuation on $K$.
    Then $$\su(K)\geq [vK:2vK]\cdot \su(Kv)\,.$$
\end{prop}
\begin{proof}
    Consider an arbitrary function field in one variable $E/Kv$.
    By \cite[Lemma 5.6]{BDGMZ}, there exists a function field in one variable $F/K$ and an extension $w$ of $v$ to $F$ such that $Fw\simeq E$, and if $E$ is nonreal, then so is $F$.
    By \Cref{P:BL}, we obtain that 
    $u(F)\geq [wF:2wF]\cdot u(Fw)$.
    By \Cref{P:val-fufi-ext-cases}, since ${Fw=E}$ and $E/Kv$ is transcendental, the quotient group $wF/vK$ is finite. 
    Hence \Cref{P:torsion-quotient-torsionfree-groups} yields that $[wF:2wF]=[vK:2vK]$.
    Therefore $$[vK:2vK]\cdot u(E) = [wF:2wF]\cdot u(Fw)\leq u(F)\leq 2\su(K)\,.$$
    Having this for all function fields in one variable $E/Kv$, the statement follows.
\end{proof}

In \Cref{T:main-body} we will see that the inequality in \Cref{P:val-su-inequality} becomes an equality when $v$ is henselian.
    
\begin{prop}\label{u-val-geq-ffiov}
Let $v$ be a valuation on $K$ and 
let $F/K$ be a function field in one variable.
Then there exists a function field in one variable $E/Kv$ such that $u(F) \geq 2\cdot [vK : 2vK]\cdot \lfloor \frac{u(E)}{2} \rfloor$, and 
if $F$ is nonreal, then $u(F) \geq [vK : 2vK]\cdot u(E)$.
\end{prop}
\begin{proof}
    We choose a valuation $w$ on $F$ extending $v$ such that 
    $Fw/Kv$ is a function field in one variable.
    (This can be done by fixing $x\in F$ transcendental over $K$, 
    taking the Gauss extension of $v$ to $K(x)$ with respect to~$x$ and then extending this valuation further to $F$.)
    By \Cref{P:val-fufi-ext-cases}, the quotient group $wF/vK$ is finite. Therefore we have $[wF : 2wF] = [vK : 2vK]$, by \Cref{P:torsion-quotient-torsionfree-groups}. 
    We now conclude by \Cref{P:BL}.
\end{proof}

The following proposition considers the behaviour of the invariant $\su$ under algebraic field extensions.
It recovers the fact that $\su(K)$ is also an upper bound on $u(K')$ for all finite field extensions $K'/K$.

\begin{prop}\label{P:odd-su}
    Let $K'/K$ be an algebraic field extension. Then $$u(K')\leq \su(K')\leq \su(K)\,.$$
    Moreover, if $K'/K$ is a direct limit of finite normal extensions of odd degree of~$K$, then $$\su(K')=\su(K)\,.$$
\end{prop}
\begin{proof}
    By \Cref{P:u-rafufi}, we have $u(K')\leq \su(K')$. 
    To show that $\su(K')\leq \su(K)$, consider
    a function field in one variable $F'/K'$ and an anisotropic torsion form $\varphi$ over $F'$.
    There exists a finitely generated subextension $F/K$ of $F'/K$ such that $\varphi$ is defined over $F$ and is a torsion form over $F$. Moreover, we may choose~$F$ to be transcendental over $K$.
    Then $F/K$ is a function field in one variable and~$\varphi$ is an anisotropic torsion form over $F$.
    Hence $\dim(\varphi)\leq u(F)\leq 2\su(K)$. 
    Since this holds for any anisotropic torsion form $\varphi$ over $F'$, it follows that $u(F')\leq 2\su(K)$.
    Having this for any function field in one variable $F'/K'$, we conclude that $\su(K')\leq \su(K)$.
    
     Assume now that $K'/K$ is a direct limit of finite normal extensions of odd degree.
    It remains to show that in this case $\su(K)\leq \su(K')$.
    Consider a function field in one variable $F/K$ and an anisotropic torsion form $\varphi$ over $F$.
    There exists a function field in one variable $F'/K'$ such that $F\subseteq F'$ and $F'=K'F$.
    Since $K'/K$ is a direct limit of finite normal extensions of odd degree, so is $F'/F$.
    Therefore, Springer's Theorem for odd degree extensions \cite[Chap.~VII, Theorem~2.7]{Lam05} yields that $\varphi$ remains anisotropic over $F'$. Since it also remains a torsion form over $F'$, we obtain that $\dim(\varphi)\leq u(F')\leq 2\su(K')$. Since this holds for all anisotropic torsion forms $\varphi$ over $F$, we see that $u(F)\leq 2\su(K')$.
    Having this for all function fields in one variable $F/K$, we conclude that  $\su(K)\leq \su(K')$.
\end{proof}

For a general algebraic extension $K'/K$, a strict inequality $\su(K')<\su(K)$ may occur.
For example, if $K'$ is an algebraic closure of $K$ while $K(\sqrt{-1})$ has a finite field extension of even degree, then \Cref{P:herquacl} yields that $\su(K')=1<\su(K)$.

\begin{qu}
    Does $\su(K')=\su(K)$ hold for every finite field extension $K'/K$?  
\end{qu}

In most cases where we have any knowledge on the $u$-invariant of function fields over $K$, we have in particular that $u(K(X,Y))=4u(K)$.

\begin{cor}\label{C:sharp-su-forallffiov}
    Assume that $\car(K)\neq 2$ and $u(K(X,Y))=4u(K)$. Then $\su(K)=u(K)$. 
\end{cor}
\begin{proof}
    By \Cref{P:odd-su}, we have $u(K)\leq \su(K)$.
To show the opposite inequality, consider an arbitrary function field in one variable $F/K$.
Then $F$ is isomorphic to a finite field extension of $K(X)$, so that \Cref{P:u-rafufi} and the hypothesis yield that $u(F) \leq \frac{1}{2}u(K(X, Y)) = 2u(K)$.
This shows that $\su(K) \leq u(K)$.
\end{proof}

The inequality $u(K(X))\geq 2 u(K)$ has an analogue for the strong $u$-invariant.

\begin{prop}
We have $\su(K(X))\geq 2\su(K)$.
\end{prop}
\begin{proof}
    Consider a function field in one variable $F/K$.
    Then $F(X)/K(X)$ is a function field in one variable and $u(F(X))\geq 2u(F)$, by \Cref{P:u-rafufi}. 
    Therefore we have $2\su(K(X))\geq u(F(X))\geq 2 u(F)$.
    Having this for every function field in one variable $F/K$ yields the claimed inequality.
\end{proof}

\begin{qu}
   Is $\su(K(X))=2\su(K)$?
\end{qu}

For the $u$-invariant, the corresponding question would have a negative answer.

\begin{ex}
    Let $K$ denote the quadratic closure of $\qq$.
    Choose an arbitrary  proper finite field extension $L/K$ (e.g.~$L=K(\sqrt[5]{6})$).
    Then $L$ is nonreal, and by \cite[Chap.~VII, Cor.~7.11]{Lam05} we have 
    $|\scg{L}|=\infty$. In particular $u(L)\geq 2$.
    It follows by \Cref{P:u-rafufi} that $u(K(X))\geq 2u(L)\geq 4>2=2u(K)$.
\end{ex}

One may  wonder whether $\su(K)$ is characterised in terms of the $u$-invariant by a single field extension of $K$. For example, is $u(F)\leq u(K(X))$ for every function field in one variable $F/K$? In other terms:

\begin{qu}
Is $\su(K)=\frac{1}2u(K(X))$?    
\end{qu}

\section{Function fields in one variable over a henselian valued field} 
\label{S:last}

We will now prove our main result and subsequently present  various situations where it can be applied.
We need a last step of preparation. 
For a valuation $w$ on a field $F$, we denote by $F^w$ the corresponding completion.
Let $K$ be a field with $\car(K)\neq 2$.

\begin{lem}\label{L:hens-fufi-ubound}
    Let $v$ be a henselian rank-$1$ valuation on $K$ and $\car(Kv)\neq 2$.
    Let~${F/K}$ be a function field in one variable and $w$ a rank-$1$ valuation on $F$ with $\mc{O}_v\subseteq\mc{O}_w$.
    Then $$u(F^w)= [wF:2wF]\cdot u(Fw)\leq [vK:2vK]\cdot 2\su(Kv)\,.$$
\end{lem}

\begin{proof}
    Since $2\in\mg{\mc{O}}_v\subseteq \mg{\mc{O}}_w$, we have that $\car(Fw)\neq 2$.
    Note that    
    $w$ extends naturally to a complete valuation $w'$ on $F^w$ with $w'F^w=wF$ and $F^ww'=Fw$. Hence $w'$ is a complete rank-$1$ valuation on $F$ and $\car(F^ww')\neq 2$. It follows by \cite[Theorem 1.3.1  (Hensel's Lemma)]{EP05} that $1+\mfm_{w'}\subseteq \sq{(F^w)}$.
    The equality on the left follows now from \Cref{P:u-hens} applied to $(F^w,w')$.
    It remains to show that $[wF : 2wF] \cdot u(Fw)\leq [vK:2vK]\cdot 2\su(Kv)$.

    The situation where $w|_K$ is trivial will be treated below as case $(4)$.
    We first consider the situation where $w|_K$ is nontrivial.
    Then $\mc{O}_v\subseteq \mc{O}_{w}\cap K\subsetneq K$.
    Since~$v$ has rank $1$, this implies that $w|_K$ is equivalent to $v$.
    We may then assume that $v=w|_K$.
    We follow the case distinction given in \Cref{P:val-fufi-ext-cases}. 

    $(1)$ Suppose that $Fw/Kv$ is a finite field extension and $wF/vK$ is a finitely generated $\zz$-module with $\rrk(Fw/Kv)=1$.
    It follows by \Cref{C:torsion-quotient-torsionfree-groups} that $[wF:2wF]=2\cdot [vK:2vK]$.
    By \Cref{P:odd-su} we have that $u(Fw)\leq \su(Kv)$. 
    From this the claimed inequality follows.

    $(2)$ Suppose that $Fw/Kv$ is a function field in one variable and $wF/wK$ is finite.
    Then by \Cref{P:torsion-quotient-torsionfree-groups}, $[wF:2wF]=[vK:2vK]$, and $u(Fw)\leq 2\su(Kv)$, which implies the claimed inequality.
    
    $(3)$ Suppose that $Fw/Kv$ is an algebraic field extension and $wF/vK$ is a torsion group.
    Then $u(Fw)\leq \su(Kv)$ and $[wF:2wF]\leq [vK:2vK]$ by \Cref{P:torsion-quotient-torsionfree-groups}.
    In this case we obtain even that $[wF:2wF]\cdot u(Fw)\leq [vK:2vK]\cdot \su(Kv)$.

    $(4)$ Assume finally that $w|_K$ is trivial.
    In this case $w$ is a discrete valuation, so we may assume that $wF=\zz$, and 
    the residue field $Fw$ is a finite extension of~$K$.
    Hence $v$ extends to a valuation $v'$ on $Fw$, and as $v$ is henselian, so is~$v'$.
    It follows by \Cref{P:u-hens} that $u(F^w) = [wF : 2wF] \cdot u(Fw) = 2u(Fw)$
    and $u(Fw)=[v'(Fw):2v'(Fw)]\cdot u((Fw)v')$.
    Since the extension $(Fw)v'/Kv$  is finite, we have that $u((Fw)v')\leq \su(Kv)$, by \Cref{P:odd-su}.
    Furthermore, $[v'(Fw):vK]\leq [Fw:K]<\infty$, whence $[v'(Fw):2v'(Fw)]=[vK:2vK]$ by \Cref{P:torsion-quotient-torsionfree-groups}.
    Thus
    $u(Fw)= [vK:2vK]\cdot u((Fw)v')\leq [vK:2vK]\cdot \su(Kv)$.
\end{proof}

We come to our generalisation of \cite[Theorem 3]{Schei09} and \cite[Cor.~6.4]{BGVG14}. 
The proof crucially depends on the local-global principle for isotropy from \cite[Cor.~3.9]{Meh19}, or more precisely, on the slightly extended version provided in \cite[Theorem 4.4]{BDGMZ}, where the valuation on the base field is assumed of rank $1$ and henselian, but not necessarily complete.

\begin{thm}\label{T:main-body}
    Let $v$ be a henselian valuation on $K$ with $\car(Kv)\neq 2$. Then $$\su(K)=[vK:2vK]\cdot \su(Kv).$$
    Furthermore, if $u(E) = 2\su(Kv)$ for every function field in one variable $E/Kv$, then $u(F) = 2[vK : 2vK] \cdot \su(Kv)$ for every function field in one variable $F/K$.
\end{thm}

\begin{proof}
By \Cref{P:val-su-inequality} we have that $\su(K) \geq [vK : 2vK] \cdot \su(Kv)$.
We now show the converse inequality.
    
    We first consider the case where $v$ has rank $1$.
    Consider an arbitrary function field in one variable $F/K$ and an anisotropic torsion form $\varphi$ over $F$. We claim that $\dim(\varphi)\leq [vK:2vK]\cdot 2\su(Kv)$.
    If $\dim (\varphi)\leq 2$, then this is trivial, because $\su(Kv)\geq 1$. 
    Suppose now that $\dim (\varphi)\geq 3$. 
    Since $\varphi$ is anisotropic over $F$ and $v$ is henselian of rank $1$, it follows by \cite[Theorem 4.4]{BDGMZ} that there exists a rank-$1$ valuation $w$ on $F$ with $\mc{O}_v\subseteq \mc{O}_w$ and such that $\varphi$ remains anisotropic over $F^w$.
    We obtain by \Cref{L:hens-fufi-ubound} that $\dim(\varphi)\leq [vK:2vK]\cdot 2\su(Kv)$, as claimed.
    Having this established for all function fields $F/K$ and all anisotropic torsion forms $\varphi$ over $F$, we conclude that $\su(K)=[vK:2vK]\cdot\su(Kv)$. 
    We have now established that $\su(K)=[vK:2vK]\cdot\su(Kv)$ whenever $v$ is of rank~$1$.

    In the next step, we assume that $v$ has arbitrary finite rank $n$ and show the statement in this situation.
    If $n=0$, then $v$ is trivial, whereby $Kv=K$, $[vK:2vK]=1$, and hence trivially $\su(K)=\su(Kv)=[vK:2vK]\cdot \su(Kv)$. The case where $n=1$ is already established, and we will use it in the induction step.
    Assume that $n>0$.
    Then the value group $vK$ contains a unique maximal proper convex subgroup $\Delta$.
    It follows that the quotient group $vK/\Delta$ is a torsion-free group of rank $1$, $\Delta$ has rank $n-1$, and the ordering on $vK$ induces an ordering on $\Gamma=vK/\Delta$. Mapping $x\in K$ to $v(x)+\Delta$ defines a valuation $v'$ on $K$ of rank $1$ with $v'K=\Gamma$. 
    The residue field $Kv'$ of $v'$ carries a valuation $\ovl{v}$ with value group $\Delta$ and hence of rank $n-1$. Furthermore, $\mc{O}_v\subseteq \mc{O}_{v'}$, and $\mfm_{v'}$ is an ideal of $\mc{O}_v$ such that $\mc{O}_v/\mfm_{v'}\simeq \mc{O}_{\ovl{v}}$. 
    Since $v$ is henselian, the valuations $v'$ on $K$ and $\ovl{v}$ and $Kv'$ are both henselian.
    By the induction hypothesis, we obtain that $\su(Kv')= [\ovl{v}(Kv'):2\ovl{v}K'(Kv')]\cdot \su((Kv')\ovl{v})$,
    and by the case where $n=1$, we have that $\su(K) = [v'K:2v'K] \cdot \su({Kv'})$.
    Furthermore, $(Kv')\ovl{v}=Kv$, $v'K=\Gamma$ and $\ovl{v}(Kv')=\Delta$.
    Therefore $\su(K)= [\Gamma:2\Gamma]\cdot \su(Kv')=[\Delta:2\Delta]\cdot [\Gamma:2\Gamma]\cdot \su(Kv)$.
    Finally, since $vK$ and $vK/\Delta$ are torsion-free, it readily follows that $[vK:2vK]=[\Gamma:2\Gamma]\cdot[\Delta:2\Delta]$. 
    We conclude that $\su(K)=[vK:2vK]\cdot \su(Kv)$, as desired.
    
    We now consider the general case without any assumption on the value group~$vK$.
    Let $\ff$ denote the prime field of $K$.
    Consider an arbitrary function field in one variable $F/K$ and an anisotropic torsion form $\varphi$ over $F$.
    We fix a positive integer $m$ such that $m\times \varphi$ is hyperbolic.
    We choose a finite subset $S$ of $F$ such that~$\varphi$ is defined over $K(S)$ and is a torsion form over $K(S)$, and further such that
    $F=K(S)$ and $\ff(S)/\ff(S\cap K)$ is of transcendence degree $1$.
    For any subfield $F_1$ of $F$ containing $S$, we have $u(F_1)\geq \dim(\varphi)$ because $\varphi$ is an anisotropic torsion form over $F_1$.
    
    We set $K_0=\ff(S\cap K)$. 
    Since the absolute transcendence degree of $K_0$ is bounded by $|S|$, it follows by \cite[Theorem~3.4.3]{EP05} that $\rk(v|_{K_0})\leq |S|+1<\infty$.
    By \Cref{L:hens-subfield-reduction}, there exists a subfield $K_1$ of $K$ containing $K_0$ such that $vK_1/vK_0$ is torsion, $vK/vK_1$ has no elements of order $2$, $Kv/K_1v$ is algebraic and purely inseparable and $v|_{K_1}$ is henselian.
    As the extension $Kv/K_1v$ is purely inseparable and  ${\car(Kv)\neq 2}$, it is in particular a direct limit of normal  extensions of odd degree.
    We conclude by \Cref{P:odd-su} that $\su(Kv)\leq \su(K_1v)$.
    As $vK/vK_1$ has no elements of order $2$, we have $[vK_1:2vK_1]\leq [vK:2vK]$.
    Since $vK_1/vK_0$ is torsion, we conclude that $\rk(v|_{K_1})=\rk(v|_{K_0})<\infty$.
    Therefore, in view of the previous part, we have $\su(K_1)=[vK_1:2vK_1]\cdot \su(K_1v)\leq [vK:2vK]\cdot \su(Kv)$.    
    Letting $F_1=K_1(S)$, we have that $F_1/K_1$ is a function field in one variable and obtain that
    $\dim(\varphi)\leq u(F_1)\leq 2\su(K_1)\leq [vK:2vK]\cdot 2\su(Kv)$.
    Having this for any anisotropic torsion form $\varphi$ defined over any function field in one variable $F/K$, we conclude that $\su(K)\leq [vK:2vK]\cdot \su(Kv)$.
    Hence, we have established the equality $\su(K)= [vK:2vK]\cdot \su(Kv)$ in all generality.

For the last part of the statement, assume now that, for every function field in one variable $E/Kv$ we have $u(E)=2\su(Kv)$, and hence also $u(E)=u(Kv(T))$.
Consider an arbitrary function field in one variable $F/K$.
If $K$ is real, then it follows by \Cref{P:hensval-res}~$(a)$ that $Kv$ is real, whereby $Kv(T)$ is real and $u(Kv(T))$ is even.
Therefore, \Cref{u-val-geq-ffiov} yields that there exists a function field in one variable $E/Kv$ such that $u(F) \geq [vK : 2vK] \cdot u(E)$.
Since $u(E)=2\su(Kv)$ and we already established that $u(F) \leq 2\su(K) = 2[vK : 2vK] \cdot \su(Kv)$, we conclude that $u(F) = 2[vK : 2vK] \cdot \su(Kv)$.
\end{proof}

\begin{rem}\label{R:compare-Mehmeti19}
Assume that $v$ is a complete rank-$1$ valuation on $K$ such that $\car(Kv)\neq 2$ and let $r=\rrk(vK)$.
In this setting, in \cite[Theorem 4.1]{Man24} (independently from this work) the bound $\su(K)\leq 2^r\su(Kv)$ is shown, which 
strengthens \cite[Cor.~6.2]{Meh19},
where the same bound is given for the special case where $vK\simeq \zz^r$, while otherwise achieving the weaker bound $\su(K)\leq 2^{r+1}\su(Kv)$.

The equality $u(K)=[vK:2vK]\cdot u(Kv)$ obtained through \Cref{T:main-body} is an improvement to the bound $\su(K)\leq 2^r\su(Kv)$.
Indeed, using \Cref{P:torsion-quotient-torsionfree-groups} with $G = vK$ and $H = 0$, we obtain that $[vK: 2vK]\leq 2^r$. Note also that this inequality can be strict (e.g.~when $vK=\rr$, we have $r=\infty$ and $[vK:2vK]=1$).
\end{rem}

We now recall three well-known constructions of valued fields $(K,v)$ with given residue field $Kv$ and a given value group $vK$, which provides us with examples where we can apply \Cref{T:main-body}.

\begin{ex}\label{Example}
Let $k$ be a field with $\car(k) \neq 2$ and let $K$ be the field of formal Puiseux series in one variable $t$ over $k$, that is, $K=\bigcup_{r\in\nat} k(\!(t^{1/r})\!)$. 
The $t$-adic valuation on $k(\!(t)\!)$ extends uniquely to a valuation $v$ on $K$, and this extension has residue field $Kv=k$ and value group $vK=\qq$. In particular, $\rrk(vK)=1$.
Note that~$v$ is henselian, but not complete.
As $\qq$ is divisible, we have that $[vK:2vK]=1$. Hence
we obtain by \Cref{T:main-body} that $\widehat{u}(K) = \widehat{u}(k)$.
Denoting by $(K',v')$ the completion of $(K,v)$, we obtain similarly that
$\widehat{u}(K') = \widehat{u}(k)$.
\end{ex}

\begin{ex}\label{Example2}
    Let $k$ be a field with $\car(k)\neq 2$ and $\Gamma$ an ordered abelian group.
    Let $K$ denote the field of Hahn series with respect to $\Gamma$ and with coefficients in~$k$.
    (The elements of $K$ are series $\sum_{\gamma\in S} a_\gamma t^\gamma$ with well-ordered support $S\subseteq\Gamma$.)
    Then $K$ carries a natural valuation $v$ with value group $\Gamma$ and residue field $k$, and $(K,v)$ is henselian, by \cite[Lemma 2.40]{Mar18}. 
    By \Cref{T:main-body}, we obtain that $\su(K)=[\Gamma: 2\Gamma]\cdot \su(k)$. 
    \end{ex}

\begin{ex}\label{Example1}
Consider a valued field $(k,v)$ with $\car(kv)\neq 2$ and value group~$vk$ contained in $\rr$ as an ordered group.
In particular $\rk(v)\leq 1$.
Let $G$ denote the divisible closure of $vk$, which is still a subgroup of $\rr$, and in fact a $\qq$-linear subspace of $\rr$. We consider the quotient space $\rr/G$.
Let $n\in\nat$ and choose $\gamma_1,\dots,\gamma_n\in\rr$  such that $\ovl{\gamma}_1,\dots,\ovl{\gamma}_n$ are $\qq$-linearly independent in $\rr/G$.
Let $K=k(X_1,\dots,X_n)$.
By \cite[Theorem 2.2.1]{EP05}, $v$ extends to a valuation $w$ on $K$ such that $w(X_i)=\gamma_i$ for $1\leq i\leq n$ and with $Kw=kv$ and $wK\simeq vk\times\zz^n$. 
In particular, we have $[wK:2wK]=2^n\cdot [vk:2vk]$.

Let $(K',w')$ be either the henselization or the completion of $(K,w)$.
Then in particular $(K',w')$ is a henselian valued field with $w'K'=wK\simeq vk\times \zz^n$ and $K'w'=Kw=kv$.
By \Cref{T:main-body}, one obtains that $\su(K')=2^n [vk: 2vk] \cdot\su(kv)$.
\end{ex}

We now recall how examples of valuations to which \Cref{T:main-body} is applicable are obtained naturally from a classical geometric setting.

\begin{ex} 
Let $(k,v)$ be a discretely valued field, and $\mathcal{O}_v$ its valuation ring.  
Let us take a regular arithmetic surface $\mathcal{X}$ over~$\mathcal{O}_v$, i.e. a regular irreducible flat model of a regular irreducible curve $X$ over~$k$.  Let $D$ be an irreducible divisor in~$\mathcal{X}$, and $\eta_D$ its generic point. Then $\mathcal{O}_{\mathcal{X}, \eta_D}$ is a 
regular local ring of dimension~$1$, hence a discrete valuation ring. It induces a discrete valuation $v_1$ on the function field $K$ of $X$, with residue field $\kappa(D)$-the function field of $D$. One can then take a closed point $x \in D$, and similarly, it induces a discrete valuation $v_2$ on~$\kappa(D)$ of residue field~$\kappa(x)$. By composing the valuations $v_1$ and $v_2$, we obtain a rank-$2$ valuation $w$ on $K$ with value group $\zz\times\zz$ and residue field $\kappa(x)$. Let $F$ denote its henselization. 
Then $\widehat{u}(F)=4\widehat{u}(\kappa(x))$ and~$\kappa(x)/kv$ is a finite  extension of~$kv$. 
This can for instance be applied to models of $\mathbb{P}_{\mathbb{Q}_p}^{1}$ over $\mathbb{Z}_p$ for any prime number~${p \neq 2}.$
\end{ex}

\begin{ex}
Let $k$ be a field in which $-1$ is not a square. 
Let us fix coordinate functions~${X, Y}$ on $\mathbb{P}_{k}^{2}$. 
Let $D=\pi^{-1}((X^2+1))$, the divisor of $\mathbb{P}_{k}^2$ induced by the fiber over $(X^2+1)$ of the projection $\pi: \mathbb{P}_{k}^2 \rightarrow \mathbb{P}_{k}^1$, $(X,Y)\mapsto X$. 
We note that~$D$ can be identified to $\mathbb{P}_{k(i)}^1$, where $i=\sqrt{-1}$ and $k(i)$ is the residue field of $(X^2+1)$ in~$\mathbb{P}_{k}^1$.  
Its generic point $\eta_{D}$ corresponds to the ideal $(X^2+1)$ in $k[X,Y],$ so the local ring~$\mathcal{O}_{\mathbb{P}_{k}^2, \eta_D}$ is $k[X, Y]_{(X^2+1)}$, and it is a discrete valuation ring. 
Let $v_1$ be the discrete valuation it induces on~$k(X,Y)$. The corresponding residue field is~$k(i)(Y)$.

Now let us take a closed point $y$ of $D$. It is a divisor of $D$, and a codimension-$2$ point of $\mathbb{P}_{k}^2$. 
Let us choose e.g. $y=(Y^2+i)$. The local ring $\mathcal{O}_{\mathbb{P}_{k}^1, y}=k(i)[Y]_{(Y^2+i)}$ is a discrete valuation ring. Let $v_2$ denote the corresponding valuation  on $k(i)(Y)$, the fraction field of $k(i)[Y]_{(Y^2+i)}$. 
Note that its residue field is~$k(\sqrt{-i})$. 
By composing $v_1$ and $v_2$ we obtain a rank-$2$ valuation $w$ on $k(X,Y)$ with value group $\zz\times\zz$ and residue field $k(\sqrt{-i})$. Let $F$ denote the henselization of $k(X,Y)$ with respect to~$w$. Then, $\su(F)=4\su(k(\sqrt{-i})).$

This same procedure can be successively applied to higher-dimensional normal varieties which have successive normal fibrations.
\end{ex}

In the following cases, \Cref{exact-uinv-ffiov} together with \Cref{T:main-body} provide directly the exact value of $u(F)$ for every function field in one variable $F/K$.

\begin{ex}
    Let $n\in\nat$, $v$ a henselian valuation on $K$ with $[vK:2vK]=2^n$ and let $F/K$ be a function field in one variable.
    \begin{enumerate}[$(a)$]
        \item If $Kv$ is algebraically closed with $\car(Kv)\neq 2$, then $u(F)=2^{n+1}$.
        \item If $Kv$ is real closed, then $u(F)=2^{n+1}$.
        \item If $Kv$ is a finite field with odd cardinality, then $u(F)=2^{n+2}$.
        \item If $Kv=\qq_p$ for a prime number $p$, then $u(F)=2^{n+3}$.
    \end{enumerate}
\end{ex}

\section{Dyadic discrete valuations with perfect residue field}
\label{S:dyadic}

In this final section, we  take a look at the dyadic case and provide an alternative proof of \cite[Theorem 4]{PS14} which is based on the local-global principle from \cite{Meh19}.

Let $K$ be a field.
Let $n\in\nat$ and let $\varphi$ be an $n$-dimensional quadratic form over $K$.
Recall from \Cref{S:DT} that there exists a unique $n\times n$-matrix $A$ over~$K$ such that 
 $\mf{b}_\varphi(x,y)=\varphi(x+y)-\varphi(x)-\varphi(y)= xAy^\tr$ for all $x,y\in K^n$, and~$\varphi$ is called \emph{nonsingular} if $\det(A)\in\mg{K}$.
 Consider now $m\in\nat$ with $m\leq n$ and an $m$-dimensional quadratic form $\psi$ over $K$.
We call $\psi$ a \emph{subform of $\varphi$} if $\psi(X)=\varphi(X\cdot C)$ for $X = (X_1, \ldots, X_m)$ and some $m\times n$-matrix $C$ of rank $m$ over $K$, or equivalently, if there exist $K$-linearly independent vectors $e_1,\dots,e_m\in K^n$ such that $\psi(c_1,\dots,c_m)=\varphi(c_1e_1+\ldots+c_me_m)$ for all $c_1,\dots,c_m\in K$.
If $\psi$ is nonsingular, then by \cite[Prop.~7.22]{EKM08}, $\psi$ is a subform of $\varphi$ if and only if $\varphi\simeq \psi\perp\vartheta$ for some quadratic form $\vartheta$.

A $2$-dimensional quadratic form will be called a \emph{binary form}.
Given a valuation~$v$ on $K$, a binary  
form $\beta$ over $K$ will be called $v$-\emph{tame}  if $\beta\simeq c(X_1^2+X_1X_2+dX_2^2)$ for some $c\in\mg{K}$ and $d\in\mc{O}_v$ with $1-4d\in\mg{\mc{O}}_v$.
Note that any $v$-tame binary form is nonsingular.

\begin{prop}\label{P:bintamesubformcrit}
    Let $\varphi$ be a  quadratic form over $K$ and $n=\dim(\varphi)$.
    Let $v$ be a valuation on $K$.
    The following two conditions are equivalent:
\begin{enumerate}[$(i)$]
        \item\label{it:tame-binary} 
        $\varphi$ contains a $v$-tame binary subform.
        \item
        There exist $x,y\in K^n$ and $d\in\mc{O}_v$ with $1-4d\in\mg{\mc{O}}_v$ such that $\mf{b}_\varphi(x,y)=1$ and $\varphi(x)\varphi(y)=d$.
\end{enumerate}
If $\car(Kv)=2$, then $(i)$ and $(ii)$ are further equivalent to the following:
\begin{enumerate}[$(i)$]
\setcounter{enumi}{2}
\item\label{it:Schwarz} 
There exist $x,y\in K^n$  with $v(\mf{b}_\varphi(x,y))\leq \min\{v(\varphi(x)),v(\varphi(y))\}<\infty$.
\end{enumerate}
\end{prop}
\begin{proof}
 $(i \Rightarrow ii,iii)$
Assume that $\varphi$ contains a $v$-tame binary subform.
This subform is isometric to $c(X_1^2+X_1X_2+dX_2^2)$ for some $c\in\mg{K}$ and $d\in\mc{O}_v$  with $1-4d\in\mg{\mc{O}}_v$.
Hence, there exist $x,y\in K^n$ such that $\varphi(x)=c$, $\varphi(y)=cd$ and $\mf{b}_\varphi(x,y)=c$.
Then $\mf{b}_\varphi(x,c^{-1}y)=1$ and 
$\varphi(x)\varphi(c^{-1}y)=c(c^{-2}cd)=d$.
Furthermore, 
$v(\mf{b}_\varphi(x,y))=v(c)=\min\{v(\varphi(x)),v(\varphi(y))\}<\infty$.

$(ii \Rightarrow i)$
Let $d\in\mc{O}_v$ and $x,y\in K^n$ be such that $1-4d\in\mg{\mc{O}}_v$, $\mf{b}_\varphi(x,y)=1$ and $d=\varphi(x)\varphi(y)$.
For $\lambda\in K$, we have
$\mf{b}_\varphi(x,\lambda y)^2=4\varphi(x)\varphi(\lambda y)$ and hence $\lambda x\neq y$.
Thus $x$ and $y$ are $K$-linearly independent.
If $\varphi(x)=\varphi(y)=0$, then we replace $x$ by $x+y$, and 
if $\varphi(y)\neq 0= \varphi(x)$, then we switch the roles of $x$ and $y$.
Hence, we may assume that $\varphi(x)\neq 0$.
We set $c=\varphi(x)$ and $y'=cy$.
Then $x$ and $y'$ are $K$-linearly independent, $\varphi(y')=cd$ and $\varphi(x_1x+x_2y')=x_1^2c+cx_1x_2+cdx_2^2$ for any $x_1,x_2\in K$.
Therefore $c(X_1^2+X_1X_2+dX_2^2)$ is a subform of $\varphi$.

$(iii\Rightarrow ii)$ 
Here we need to assume that $\car(Kv)=2$.
Let $x,y\in K^n$ be such that $v(\mf{b}_\varphi(x,y))\leq \min\{v(\varphi(x),v(\varphi(y))\}<\infty$.
Without loss of generality, $v(\varphi(x))\leq v(\varphi(y))$ and $\varphi(x)\neq 0$.
Set $y'=\mf{b}_\varphi(x,y)^{-1}y$ and $d=\varphi(x)\varphi(y')$. Then $\mf{b}_\varphi(x,y')=1$ and $v(d)=v(\varphi(x))+v(\varphi(y))-2v(\mf{b}_\varphi(x,y))\geq 0$, whereby $d\in\mc{O}_v$.
Since $\car(Kv)=2$, it follows that $1-4d\in\mg{\mc{O}}_v$.
\end{proof}

\begin{ex}\label{isotame}
    In the setting of \Cref{P:bintamesubformcrit}, assume that $\varphi$ is nonsingular and isotropic. Hence there exist $x,y\in K^n$ such that $\varphi(x)=0$ and $\mf{b}_\varphi(x,y)=1$. Then $\varphi(x)\varphi(y)=0\in\mc{O}_v$, so \Cref{P:bintamesubformcrit} yields that $\varphi$ contains a $v$-tame binary form.
    More precisely, in this case $\varphi$ contains the hyperbolic form $\hh=X_1X_2$ as a subform, which is isometric to $X_1(X_1+X_2)$ and hence $v$-tame.
\end{ex}

A valuation $v$ on $K$ is called \emph{$2$-henselian} if every quadratic polynomial in $\mc{O}_v[X]$ having a simple root in $Kv$ also has a root in $K$.

\begin{rem}\label{R:2hens}
    If $\car(Kv)\neq 2$, then $v$ is $2$-henselian if and only if $1+\mfm_v\subseteq \sq{\mc{O}}_v$, by \cite[Cor.~4.2.4]{EP05}.
\end{rem}

The following statement clarifies the role of tame binary subforms in determining isotropy over a $2$-henselian valued field.

\begin{lem}\label{L:2hens-tame-rep-lift}
    Let $v$ be a $2$-henselian valuation on $K$.
    Let $c,d\in\mc{O}_v$ be such that $1-4d\in\mg{\mc{O}}_v$.
    If $\ovl{c}$ is nontrivially represented by $X_1^2+X_1X_2+\ovl{d}X_2^2$ over $Kv$, then~$c$ is represented nontrivially by $X_1^2+X_1X_2+dX_2^2$ over $K$.
    In particular, if $X_1^2+X_1X_2+\ovl{d}X_2^2$ is isotropic over $Kv$, then $X_1^2+X_1X_2+dX_2^2$ is isotropic over~$K$.
    \end{lem}
\begin{proof}
    Assume that $\ovl{c}$ is nontrivially represented by $X_1^2+X_1X_2+\ovl{d}X_2^2$ over $Kv$.
    Hence there exist elements $x_1,x_2\in \mc{O}_v$ which are not both contained in $\mfm_v$ such that $x_1^2+x_1x_2+dx_2^2-c\in\mfm_v$.
    If $x_2+2x_1\in\mg{\mc{O}}_v$, then $\ovl{x}_1\in Kv$ is a simple root of the polynomial $T^2+Tx_2+dx_2^2-c\in\mc{O}_v[T]$, and since $v$ is $2$-henselian, there exists $x_1'\in K$ with $x_1'^2+x_1'x_2+dx_2^2-c=0$.
    If $x_2+2x_1\in\mfm_v$, then $x_1\in\mg{\mc{O}}_v$, and as $1-4d\in\mg{\mc{O}}_v$, it follows that 
    $x_1+2dx_2\equiv (1-4d)x_1\not\equiv 0\bmod \mfm_v$, whereby the residue $\ovl{x}_2\in Kv$ is a simple root of the polynomial $x_1^2+x_1T+dT^2-c\in\mc{O}_v[T]$. As $v$ is $2$-henselian, we conclude that there exists $x_2'\in K$ with ${x_1^2+x_1x_2'+d{x_2'}^2-c=0}$.  
\end{proof}

Whether a given quadratic form possesses a tame binary subform is invariant under certain extensions of valued fields. In this context, we state two crucial reductions by certain extensions of the base field.

Following \cite{ET11}, given a valuation $v$ on $K$, a field extension $K'/K$ will be called \emph{$v$-inertial} if it is algebraic and  $v$ extends to a valuation $v'$ on $K'$ in such way that $Lv'/Kv$ is separable with $[Lv':Kv]=[L:K]$ for every finite subextension $L/K$ of $K'/K$. 

\begin{thm}[Elomary-Tignol]\label{P:ET}
    Let $v$ be a henselian valuation on $K$ with $\car(Kv)=2$ and $\varphi$ a nonsingular quadratic form over $K$.
    Then $\varphi$ contains a $v$-tame binary subform if and only if there exists a $v$-inertial extension $L/K$ such that $\varphi_{L}$ is isotropic.
\end{thm}
\begin{proof}
    If $\varphi$ is isotropic, then the equivalence holds trivially, in view of \Cref{isotame}.
    Assume that $\varphi$ is anisotropic.
    Let $n=\dim(\varphi)$.
    It follows by \Cref{P:bintamesubformcrit} that
    $\varphi$ possesses no $v$-tame binary subform if and only if we have $v(\mf{b}_\varphi(x,y))>\min\{v(\varphi(x)),v(\varphi(y))\}$ for all nonzero vectors $x,y\in K^n$, which is condition $(S')$ in \cite[Sect.~4]{ET11}. 
    Therefore the statement follows from \cite[Theorem~16]{ET11} in this case.
\end{proof}

\begin{lem}\label{L:S-density}
    Let $(K,v)$ be a valued field.
    Let $k$ be a subfield of $K$ which is dense in $K$ with respect to the $v$-adic topology. 
    Let $\varphi$ be a quadratic form over $k$. 
    Then $\varphi$ contains a $v$-tame binary subform over $K$ if and only if $\varphi$ contains a $v$-tame binary subform over $k$.
\end{lem}
\begin{proof}
Note that 
$\mc{O}_v^\#  =  \{d\in \mc{O}_v\mid 1-4d\in\mg{\mc{O}}_v\}$ is open in $K$. 
We endow $K^n\times K^n$ with the $v$-adic topology. 
Then
 $U = \{(x,y)\in K^n\times K^n\mid \mf{b}_\varphi(x,y)\neq 0\}$ is open in $K^n\times K^n$.
The map 
    $\Phi:U\to K,(x,y)\mapsto \varphi(x)\varphi(y)\mf{b}_\varphi(x,y)^{-2}$
    is continuous. 
    We set $S  =  \{\varphi(x)\varphi(y)\mid x,y\in K^n, \mf{b}_\varphi(x,y)=1\}$.
    For $(x,y)\in U$ and $\lambda=\mf{b}_\varphi(x,y)^{-1}$, we have $\mf{b}_\varphi(\lambda x,y)=1$, whereby $\Phi(x,y)=\Phi(\lambda x,y)\in S$.
    Therefore $S=\Phi(U)$.
    The hypothesis that $k$ is dense in $K$ implies that $k^n\times k^n$ is dense in $K^n\times K^n$. It follows that $(k^n\times k^n)\cap \Phi^{-1}(\mc{O}_v^\#)$ is dense in $\Phi^{-1}(\mc{O}_v^\#)$.
    By \Cref{P:bintamesubformcrit}, this yields the claimed statement.
\end{proof}

\Cref{L:S-density} can typically be applied for the completion of a valued field.
Recall that we denote the completion of a field $F$ with respect to a valuation $w$ by $F^w$.

\begin{thm}\label{inertanisobound}
Let $(K,v)$ be a discretely valued field such that $\car(K)=0$, $\car(Kv)=2$ and $Kv$ is perfect.
Let $n\in\nat$ and let $F/K$ be a field extension of transcendence degree $n$ and $\varphi$ a quadratic form over $F$ with $\dim(\varphi)>2^{n+1}$.
Let $w$ be a valuation on $F$ extending $v$ and assume that $w$ is henselian or of rank $1$.
Then $\varphi$ contains a $w$-tame binary subform over $F$.
    \end{thm}
    \begin{proof}       
       Assume first that $w$ is henselian.
        Let $(F', w')$ be a maximal inertial algebraic extension of $(F, w)$. (It is unique up to isomorphism of valued fields, and called
        the \emph{inertia field of $(F,w)$}; see e.g.~\cite[Section 5.2]{EP05}.)
        Since $w$ is henselian, so is $w'$, and it follows that $F'w'$ is separably closed.
        We choose a maximal algebraic extension $K'/K$ inside $F'/K$ with $w'K'=vK=\zz$.
        Set $v'=w'|_{K'}$.
        Since $(F',w')$ is henselian and $F'w'$ is separably closed, it follows that $K'v'$ is separably closed.
        Since $Kv$ is perfect, we obtain that $K'v'$ is algebraically closed.
        Hence $(K',v')$ is a henselian discretely valued field with algebraically closed residue field.
        Since $F'/K'$ is a field extension of transcendence degree $n$,  it follows by \Cref{T:Lang} that $u(F')\leq 2^{n+1}$.
        Hence $\varphi$ is isotropic over $F'$.
        Since $F'/F$ is a $w$-inertial extension, we conclude by \Cref{P:ET} that $\varphi$ contains a $w$-tame binary form.
        This proves the statement under the assumption that $w$ is henselian.

Let us now consider the case where $\rk (w)=1$ while $w$ is not assumed to be henselian.
Let $\hat{w}$ denote the natural valuation extension of $w$ to $F^w$.
Since $\rk(w)=1$, the valuation $\hat{w}$ is henselian.
Let $F'$ denote the relative algebraic closure of $F$ inside $F^w$ and $w'=\hat{w}|_{F'}$.
Then $(F',w')$ is henselian, too, and $F'/K$ is of transcendence degree $n$ over $K$. 
Hence it follows from the case which is already established that $\varphi$ contains a $w'$-tame binary form over $F'$. 
Since $F$ is dense in~$F^w$ with respect to the $\hat{w}$-adic topology, and hence as well in $F'$, we conclude by
 \Cref{L:S-density} that $\varphi$ contains a $w$-tame binary form over $F$.
\end{proof}

\begin{rem}
We do not know whether the conclusion of \Cref{inertanisobound} remains valid without assuming $w$ to be henselian or of rank $1$.
\end{rem}

\begin{prop}\label{P:inerquadnormsurj}
    Let $v$ be a $2$-henselian valuation on $K$ with $\car(Kv)=2$ and such that $Kv$ is perfect.
    Let $\beta$ be an anisotropic $v$-tame binary 
    form over $F$.
    Then the nonzero elements represented by $\beta$ over~$K$ form a coset of $\mg{K}/\sq{K}\mg{\mc{O}}_v$.
\end{prop}
\begin{proof}
By the hypothesis, $\beta\simeq c(X_1^2+X_1X_2+aX_2^2)$ for some  $a\in\mc{O}_v$ and ${c\in\mg{K}}$.
Since $c$ is represented by $\beta$, we need to show that the nonzero elements represented by $c(X_1^2+X_1X_2+aX_2^2)$ over $K$ are precisely the elements of $c\sq{K}\mg{\mc{O}}_v$.
For this, we may assume without loss of generality that $c=1$.
Since $v$ is  $2$-henselian and~$\beta$ is anisotropic,
by \Cref{L:2hens-tame-rep-lift}, the quadratic form $X_1^2+X_1X_2+\ovl{a}X_2^2$ over $Kv$ is anisotropic, too.
This implies that, for any $x_1,x_2\in K$, we have 
$v(x_1^2+x_1x_2+ax_2^2)=2\min\{v(x_1),v(x_2)\}$ and hence $x_1^2+x_1x_2+ax_2^2\in\{0\}\cup\sq{K}\cdot \mg{\mc{O}}_v$.
On the other hand, since $Kv$ is perfect and $\car(Kv)=2$, the quadratic form $X_1^2+X_1X_2+\ovl{a}X_2^2$ over $Kv$ represents all elements of $Kv$.
Since $v$ is $2$-henselian, we conclude by \Cref{L:2hens-tame-rep-lift} that 
$X_1^2+X_1X_1+aX_2^2$ over $K$ represents all elements of $\mg{\mc{O}}_v$,  hence also all elements of~$\sq{K} \mg{\mc{O}}_v$.
\end{proof} 

\begin{lem}\label{P:PS-preparation}
Let $v$ be a $2$-henselian valuation on $K$ such that $Kv$ is perfect, $\car(Kv)=2$ and $[vK:2vK]<\infty$. 
Let $\varphi$ be an anisotropic quadratic form over~$K$. 
Let $r\in\nat$ and let $\beta_1,\dots,\beta_r$ be $v$-tame binary 
forms over $K$ and $\psi$ a quadratic form over $K$ containing no $v$-tame binary subform and such that 
\begin{equation}
\tag{$\ast$}
\varphi\simeq \beta_1\perp\ldots\perp\beta_r\perp\psi.
\end{equation}
Set $m=[vK:2vK]$.
Then  $r\leq m$ and the nonzero elements represented by $\psi$ over $K$
    are contained in a union of $m-r$ cosets of $\mg{K}/\mg{\mc{O}}_v\sq{K}$.
\end{lem}
\begin{proof}
We denote by $B_i$ be the set of elements of $\mg{K}$ represented by $\beta_i$, for $1\leq i\leq r$.
It follows by \Cref{P:inerquadnormsurj}  that
$B_1,\dots,B_r$ are cosets of $\mg{K}/\mg{{\mc{O}}}_v\sq{K}$.
Note that $\mg{K}/\mg{\mc{O}}_v\sq{K}\simeq vK/2vK$, whereby $m=|\mg{K}/\mg{\mc{O}}_v\sq{K}|$.
If $B_i=B_j$ for some distinct $i,j\in\{1,\dots,r\}$, 
then $\beta_i\perp\beta_j$ is isotropic.
If $\psi$ represents an element of $B_i$ for some $i\in\{1,\dots,n\}$, then $\beta_i\perp\psi$ is isotropic.
Hence in view of $(\ast)$, in either of these two situations, $\varphi$ is isotropic.
In the remaining case, the cosets $B_1,\dots,B_r$ are distinct, and 
the elements of $\mg{K}$ represented by $\psi$ are contained in $\mg{K}\setminus\bigcup_{i=1}^rB_i$, which is a union of $m-r$ cosets of $\mg{K}/\mg{\mc{O}}_v\sq{K}$.
\end{proof}

Recall that, by \cite[Prop.~7.22]{EKM08}, given a nonsingular subform $\beta$ of a quadratic form $\varphi$, we obtain that $\varphi\simeq \beta\perp \varphi'$ for another quadratic form $\varphi'$.
Therefore, given a valuation $v$ on $K$ and a quadratic form $\varphi$ over $K$, a decomposition of $\varphi\simeq \beta_1\perp\ldots\perp\beta_r\perp\psi$ as in \Cref{P:PS-preparation} can always be found, without any  assumptions on~$v$, using that $v$-tame binary forms are  nonsingular.

\begin{thm}[Parimala-Suresh]\label{PS}
        Let $(K,v)$ be a henselian discretely valued field with $\car(K)=0$, $\car(Kv)=2$ and $Kv$ perfect.
        Then $\su(K)=2\su(Kv)$  with $$\su(Kv)= \begin{cases}
         1 &\text{if } Kv \text{ has no finite extension of even degree, }\\ 2 &\text{otherwise.}\end{cases}$$ 
Furthermore, $u(F)=4\su(Kv)$ for every function field in one variable $F/K$.
\end{thm}
\begin{proof}
Consider first the case where $Kv$ has no finite extension of even degree. Then $\su(Kv) = 1$, by \Cref{P:herquacl}, using that $-1=1\in\sq{(Kv)}$.
It further follows that there exists an inertial extension of valued fields $(K', v')/(K,v)$ which is a direct limit of odd degree normal finite extensions and such that $K'v'$ is algebraically closed.
Since $(K',v')$ is henselian as well, it follows by \Cref{T:Lang} that $\su(K') = 2$. We conclude by \Cref{P:odd-su} that $\su(K) = 2$.
By \Cref{P:herquacl}, this implies that $u(F) = 4$ for every function field in one variable $F/K$.

We now assume that there exists a finite field extension of even degree  $\ell/Kv$.
By \Cref{P:herquacl} we have that $u(E) \geq 4$ for every function field in one variable $E/Kv$, and by \Cref{u-val-geq-ffiov} this yields that $u(F) \geq 8$ for every function field in one variable $F/K$.
It therefore remains only to show that $\su(K) \leq 4$.

Consider a function field in one variable $F/K$ and an anisotropic quadratic form $\varphi$ over $F$.
We need to show that $\dim(\varphi)\leq 8$.

By \cite[Theorem 4.4]{BDGMZ}, there exists a rank-$1$ valuation $w$ on $F$ with $\mc{O}_v\subseteq \mc{O}_w$ and such that $\varphi$ remains anisotropic over $F^w$.
In particular $\dim(\varphi)\leq u(F^w)$.

Consider first the case where $K\subseteq \mc{O}_w$. 
Then $\car(Fw)=\car(K)=0$, and since $F/K$ is a function field in one variable, $w$ is discrete.
It follows that $w$ extends naturally to a discrete valuation on $F^w$ with residue field $Fw$.
Therefore $u(F^w)=2u(Fw)$, in view of \Cref{P:u-hens}.
Since $Fw/K$ is a finite field extension,
$v$ extends to a discrete valuation $v'$ on $Fw$. Since $v$ is henselian and $Kv$ is perfect, we obtain that $v'$ is henselian and $(Fw)v'$ is perfect. 
Since $\car((Fw)v')=\car(Kv)=2$, we obtain by
\cite[Cor. 1]{MMW91} that $u((Fw)v') \leq 2$.
By \Cref{T:MMW} it follows that $u(Fw)=2u((Fw)v') \leq 4$. 
Therefore, $\dim(\varphi)\leq u(F^w)\leq 8$.

Hence we are left with the case where $K$ is not contained in $\mc{O}_w$. Since the valuation $v$ on $K$ is discrete, $\mc{O}_v$ is a maximal proper subring of $K$, so we obtain that
$K\cap\mc{O}_w=\mc{O}_v$.
Hence, after rescaling values, $w$ is an extension of $v$ from~$K$ to $F$.
Then $Fw/Kv$ is either an algebraic extension or a function field in one variable.
Since $Kv$ is perfect  and $Fw/Kv$ has transcendence degree at most $1$, we have $[Fw : (Fw)^{(2)}] \leq 2$, by \cite[Satz 15]{Tei36}, where $(Fw)^{(2)}$ is the subfield of~$Fw$ of squares.
It follows by \cite[Cor.~1]{MMW91} that $u(Fw)\leq 2[Fw : (Fw)^{(2)}] \leq 4$.
If $w$ is discrete,
 then we obtain using \Cref{T:MMW} that $u(F^w)=2u(Fw)\leq 8$, whereby
 $\dim(\varphi)\leq u(F^w)\leq 8$.

It remains to consider the case where $w$ is not discrete.
In this case, it follows by 
\Cref{P:val-fufi-ext-cases} that $Fw/Kv$ is an algebraic extension and the group $wF/vK$ is either torsion or finitely generated with $\rrk(wF/vK)=1$.
In particular, $Fw$ is perfect and, since $vK\simeq \zz$, we obtain that $[wF:2wF]\leq 4$.
We write $$\varphi\simeq \beta_1\perp\ldots\perp\beta_r\perp\psi$$ for some $r\in\nat$, $w$-tame binary 
forms $\beta_1,\ldots,\beta_r$ over $F$ and a quadratic form $\psi$ over $F$ that contains no $w$-tame binary subform.
Since $w$ has rank $1$,  \Cref{inertanisobound} yields that $\dim(\psi)\leq 4$.
    Therefore $$\dim(\varphi)=2r+\dim(\psi)\leq 2r+4\,.$$
We set $m=[wF:2wF]$ and obtain from  \Cref{P:torsion-quotient-torsionfree-groups} that $m\leq 4$.
Since~$\varphi$ is anisotropic over $F^w$, we obtain  from  \Cref{P:PS-preparation}  that $r\leq m\leq 4$ and the nonzero elements represented by $\psi$ over $F^w$ are contained in a union of $m-r$ cosets of $\mg{F^w}/\mg{\mc{O}}_w\sq{F^w}$. 
In particular, if $r=4$, then $\psi$ is trivial.

Hence $\dim(\varphi)\leq 10$. Moreover, if $r\leq 2$ or $r=4$, then we get that $\dim(\varphi)\leq 8$, as desired.
Consider now the case where $r= 3$ and $\psi$ is nontrivial. 
Then $m=4$,
$\dim(\varphi)=6+\dim(\psi)\leq 10$ and 
the nonzero elements represented by $\psi$ over $F^w$ are contained in a single coset of 
$\mg{F^w}/\mg{\mc{O}}_w\sq{F^w}$.
We fix $c_1,\dots,c_4\in\mg{F}$ such that $w(c_i)\not\equiv w(c_j)\bmod 2wF$ for $1\leq i<j\leq 4$, whereby they are representative of the distinct cosets of $\mg{F^w}/\mg{\mc{O}}_w\sq{F^w}$.
We consider the quadratic form
$$\varphi'=\psi\perp c_1\psi\perp c_2\psi\perp c_3\psi$$ over $F$.
Then $\dim(\varphi')=4\dim(\psi)$ and $\varphi'$ is anisotropic over $F^w$.
The above argument therefore yields that 
$\dim(\varphi')\leq 10$. 
Hence $\dim(\psi)\leq 2$, so that also in this case we obtain that $\dim(\varphi)\leq 8$.
\end{proof}

\subsection*{Acknowledgments}

We would like to thank David Leep for inspiring discussions including various pointers to the literature, in particular concerning \Cref{S:DT}, as well as Philip Dittmann for several comments and in particular helping us to make \Cref{L:hens-subfield-reduction} correct.
Karim Johannes Becher acknowledges support by the \emph{Bijzonder Onderzoeksfonds} (project BOF-DOCPRO4, Nr. 49084, `\emph{Constructive proofs for quadratic forms}').
Nicolas Daans gratefully acknowledges support by \emph{Czech Science Foundation} (GA\v CR) grant 21-00420M, \emph{Charles University} PRIMUS Research Programme PRIMUS/24/SCI/010, and \emph{Research Foundation - Flanders (FWO)}, fellowship 1208225N.

\bibliographystyle{plain}

\end{document}